\documentclass[a4paper,11 pt,reqno,twoside]{amsart}
\usepackage{a4wide}
\usepackage{amssymb}
\usepackage{latexsym}
\usepackage{amsmath}

\usepackage{geometry}
\usepackage[UKenglish]{babel}
\usepackage{color}
\usepackage[bookmarks=true,
            bookmarksnumbered=true,
            colorlinks=false,
            pdfstartview=FitH]{hyperref}
 \hypersetup{pdftitle={Quantum stochastic Lie--Trotter product formula II},
 pdfauthor={J. Martin Lindsay}}
 \usepackage{comment} %added

\theoremstyle{plain}
\newtheorem{propn}{Proposition}[section]
\newtheorem{thm}[propn]{Theorem}
\newtheorem{lemma}[propn]{Lemma}
\newtheorem{cor}[propn]{Corollary}
\newtheorem{conjecture}[propn]{Conjecture}

\theoremstyle{definition}
\newtheorem{defn}[propn]{Definition}
\newtheorem{eg}[propn]{Example}

\theoremstyle{remark}
\newtheorem*{rem}{Remark}
\newtheorem*{rems}{Remarks}
 \newcommand{\Xh}{X^{\la h \ra}}

 \newcommand{\ud}{\,\mathrm{d}}
 
\newcommand{\QSChol}{\mathbb{QSC}_{\hol}}
  \newcommand{\Shol}{\mathcal{S}_{\hol}}

  \newcommand{\Upsilonen}{\Upsilon^n}
 \newcommand{\Upsilonone}{\Upsilon^1}
 \newcommand{\Upsilontwo}{\Upsilon^2}
 \newcommand{\FD}{F^\Delta}
 \newcommand{\components}{(\gamma, L, \wt{L}, C)}
 \newcommand{\componentsi}{(\gamma_i, L_i, \wt{L}_i, C_i)}
 \newcommand{\QF}{\mathcal{Q}(\hil \op \Hil)}
  \newcommand{\Ltwoloc}{L^2_{\loc}}
 \newcommand{\Expect}{\mathbb{E}}
 \newcommand{\comp}{\lhd}

 \newcommand{\Qblg}{\mathfrak{b}(\khat\ot\init)}
 \newcommand{\Qzlg}{\mathfrak{z}(\khat\ot\init)}
 \newcommand{\Qclg}{\mathfrak{c}(\khat\ot\init)}
 \newcommand{\Qqclg}{\mathfrak{qc}(\khat\ot\init)}
 \newcommand{\Qqclgbeta}{\mathfrak{qc}_\beta(\khat\ot\init)}
 \newcommand{\Qqclgbetai}{\mathfrak{qc}_{\beta_i}(\khat\ot\init)}
 \newcommand{\Qqclgbetaone}{\mathfrak{qc}_{\beta_1}(\khat\ot\init)}
 \newcommand{\Qqclgbetatwo}{\mathfrak{qc}_{\beta_2}(\khat\ot\init)}
 \newcommand{\Qilg}{\mathfrak{i}(\khat\ot\init)}
 \newcommand{\Qulg}{\mathfrak{u}(\khat\ot\init)}
 
 \newcommand{\Qpglg}{\mathfrak{pg}(\khat\ot\init)}
 \newcommand{\Qglg}{\mathfrak{g}(\khat\ot\init)}
  \newcommand{\Qwnglg}{\mathfrak{wn}\text{-}\mathfrak{g}(\khat\ot\init)}
 
 \newcommand{\Qpplg}{\mathfrak{pp}(\khat\ot\init)}

 \newcommand{\QpglgSolo}{\mathfrak{pg}}

 \newcommand{\QwnglgSolo}{\mathfrak{wn}\text{-}\mathfrak{g}}

 \newcommand{\QqclgC}{\mathfrak{qc}(\khat)}

 \newcommand{\blg}{\mathfrak{b}(\hil\oplus\Hil)}
 \newcommand{\blgone}{\mathfrak{b}(\hil\oplus\Hil_1)}
 \newcommand{\blgtwo}{\mathfrak{b}(\hil\oplus\Hil_2)}
 \newcommand{\zlg}{\mathfrak{z}(\hil\oplus\Hil)}
 \newcommand{\clg}{\mathfrak{c}(\hil\oplus\Hil)}
  \newcommand{\clgone}{\mathfrak{c}(\hil\oplus\Hil_1)}
 \newcommand{\qclg}{\mathfrak{qc}(\hil\oplus\Hil)}
  
  \newcommand{\qclgonebeta}{\mathfrak{qc}_\beta(\hil\oplus\Hil_1)}
 \newcommand{\qclgbeta}{\mathfrak{qc}_\beta(\hil\oplus\Hil)}

 \newcommand{\ilg}{\mathfrak{i}(\hil\oplus\Hil)}
 \newcommand{\ulg}{\mathfrak{u}(\hil\oplus\Hil)}
 \newcommand{\ulgprime}{\mathfrak{u}(\hil\oplus\Hil')}

\newcommand{\Domain}{\mathcal{D}}

\newcommand{\Expectation}{\mathbb{E}}

 \newcommand{\Vone}{V^{(1)}}
  \newcommand{\Vtwo}{V^{(2)}}

 \newcommand{\Hone}{H^{(1)}}
  \newcommand{\Htwo}{H^{(2)}}
   \newcommand{\Hthree}{H^{(3)}}

 \newcommand{\Yone}{Y^{(1)}}
  \newcommand{\Ytwo}{Y^{(2)}}
   \newcommand{\Ythree}{Y^{(3)}}
    \newcommand{\Yi}{Y^{(i)}}
 
 \newcommand{\Zone}{Z^{(1)}}
  \newcommand{\Ztwo}{Z^{(2)}}
   \newcommand{\Zthree}{Z^{(3)}}
    \newcommand{\Zi}{Z^{(i)}}

\newcommand{\Xfourhol}{\mathfrak{X}^4_{\hol}}
 \newcommand{\Xhol}{\mathfrak{X}_{\hol}}

 \newcommand{\dualco}{V^{\sharp}}
 \newcommand{\reversedco}{V^{\reversed}}
 \newcommand{\dual}{\sharp}

\newcommand{\otul}{\underline{\ot}\,}
\newcommand{\otol}{\overline{\ot}\,}

\newcommand{\ve}{\varepsilon}

\newcommand{\vp}{\varpi}

\newcommand{\Hil}{\mathsf{H}}
\newcommand{\hil}{\mathsf{h}}
\newcommand{\Kil}{\mathsf{K}}

\newcommand{\Al}{\mathsf{A}}

\newcommand{\init}{\mathfrak{h}}
\newcommand{\noise}{\mathsf{k}}
\newcommand{\khat}{\wh{\noise}}
\newcommand{\Fock}{\mathcal{F}}
\newcommand{\Exps}{\mathcal{E}}

\newcommand{\Step}{\mathbb{S}}

\newcommand{\Real}{\mathbb{R}}
\newcommand{\Rplus}{\Real_+}
\newcommand{\Comp}{\mathbb{C}}
\newcommand{\Nat}{\mathbb{N}}

\newcommand{\ip}[2]{\langle #1, #2 \rangle}

\newcommand{\norm}[1]{\lVert #1 \rVert}

\newcommand{\bra}[1]{\langle #1 |}
\newcommand{\ket}[1]{| #1 \rangle}

\newcommand{\sa}{{\text{\tu{sa}}}}

\newcommand{\wh}{\widehat}
\newcommand{\wt}{\widetilde}
\newcommand{\ol}{\overline}
\newcommand{\ot}{\otimes}

\newcommand{\op}{\oplus}

\newcommand{\les}{\leqslant}
\newcommand{\ges}{\geqslant}

\newcommand{\tu}{\textup}

\DeclareMathOperator{\reversed}{r} %%% new
 \DeclareMathOperator{\Gauss}{g} %%% new
 \DeclareMathOperator{\wnGauss}{wn-g} %%% new
 \DeclareMathOperator{\mGauss}{mg} %%% new
 \DeclareMathOperator{\Pres}{p} %%% new

 \DeclareMathOperator{\loc}{loc}

 \DeclareMathOperator{\hol}{hol}
  
\DeclareMathOperator{\Dom}{Dom}
\DeclareMathOperator{\Ran}{Ran}
\DeclareMathOperator{\Lin}{Lin}

\DeclareMathOperator{\Ker}{Ker}
\DeclareMathOperator{\id}{id}
\DeclareMathOperator{\re}{Re}
\DeclareMathOperator{\im}{Im}

\DeclareMathOperator{\LHS}{LHS}
\DeclareMathOperator{\RHS}{RHS}

\newcommand{\Linbar}{\ol{\Lin}}
\newenvironment{alist}
{

\begin{enumerate}}
{\end{enumerate}}

\newenvironment{rlist}
{

\begin{enumerate}}
{\end{enumerate}}

\newcommand{\la}{\langle}
\newcommand{\ra}{\rangle}

\numberwithin{equation}{section} 

\begin{document}

\title[QS Lie--Trotter product formula]
{Quantum stochastic \\
Lie--Trotter product formula II}
\author[Martin Lindsay]{J.\ Martin Lindsay}
\address{Department of Mathematics \& Statistics \\
Lancaster University \\ Lancaster LA1 4YF \\ UK}
\email{j.m.lindsay@lancs.ac.uk}

 \subjclass[2000]
 {46L53 %% Noncommutative probability and statistics
 (primary);
  81S25,  %% Quantum stochastic calculus
  47D06, %%  One-parameter semigroups and linear evolution equations
 46N50  %% Functional analysis - applications in quantum physics
 (secondary).}
 %; Secondary 46L07, 47D06} MAYBE ADD SOME MORE - semigroup theory

 \keywords{Lie--Trotter product formula,
quantum stochastic cocycle,
one-parameter semigroup,
series product, concatenation product,
quantum It\^o algebra,
quantum stochastic analysis}

\begin{abstract}
 A natural counterpart to the Lie--Trotter
 product formula for 
 norm-contin\-uous 
 %norm-continuous
 one-parameter semigroups is
 proved, for the class of
quasicontractive quantum stochastic operator cocycles
whose expectation semigroup is norm continuous.
 Compared to previous such results, the assumption of a strong form
 of  independence of the constituent cocycles is
 overcome.
 The analysis is facilitated by the development of some
 quantum It\^o algebra.
It is also shown how
the maximal Gaussian component of a  quantum stochastic generator may be extracted
--- leading to a canonical decomposition of such generators,
and
 the connection to perturbation theory is described.
Finally,
the quantum It\^o algebra is extended to quadratic form generators,
and
a conjecture is formulated for the extension of the product formula
 to holomorphic quantum stochastic cocycles.
\end{abstract}

\maketitle

%%%%%%%%%%%%%%%%%%%%%%%%%%%%%%%%%%%      TABLE OF CONTENTS
 \tableofcontents

 \section*{Introduction}
 \label{section: introduction}

 The Lie product formula in a unital Banach algebra states that
 \[
 ( e^{a / n}  e^{b / n} )^n \to  e^{a+b}
 \text{ as } n \to \infty.
 \]
 Trotter extended this to $C_0$-semigroups on a Banach
 space where it holds under compatibility assumptions on the
 generators, convergence being in the strong operator sense
 (\cite{Tr2}, see \emph{e.g.}~\cite{Davies}).
 It has been further refined,
 notably by Chernoff (\cite{Chernoff}) and Kato (\cite{Kato T}).
 These product formulae are widely used in mathematical physics
 and probability theory -- for example in establishing positivity
 preservation of semigroups, and they have an intimate connection to Feynman--Kac
 formulae (see \emph{e.g.}~\cite{ReedSimon}).
 Given that quantum stochastic
 cocycles may be analysed from their associated semigroups
 (\cite{LiW}), it is natural to seek product formulae in this
 context.
 Further motivation comes from the fact that
 such cocycles are quantum counterparts to stochastic
 semigroups in the sense of Skorohod (\cite{Skorohod}).
 Product formulae have been obtained in a variety of quantum stochastic
 settings (\cite{PaS}, [LS$_{1,3}$], \cite{DLT}, \cite{DGS}).
 The earliest of these dates from before the advent of quantum
 stochastic calculus (\cite{HuP}).
 In all of these works the constituent cocycles enjoy a strong
 independence property, namely their respective noise dimension
 spaces are mutually orthogonal.

 In this paper a Lie--Trotter
 product formula is established for quasicontractive
elementary (\emph{i.e.} Markov-regular) quantum stochastic operator cocycles,
 with no independence assumption on the driving quantum noise.
 It is a direct generalisation of the product formula proved in~\cite{LS L-T},
 and is proved by quite different means.
 Properties of the composition law on the set of
 quantum stochastic generators that is realised by
 the stochastic product formula established here are also elucidated.
Known in the setting of quantum control theory as the series product (\cite{GoughJames}),
 it is more commonly associated with the \emph{perturbation} of quantum stochastic cocycles (\cite{EvansHudson}).
 The composition of stochastic generators also corresponds to
 the operator product (\emph{i.e.} standard composition) of the generators of the quantum random walks
 whose scaled embeddings approximate the constituent cocycles (\cite{BGL}).
Analysis of this composition leads to left and right series decompositions of
a quantum stochastic generator.
A  decomposition for such generators in terms of
the so-called concatenation product is also given;
this yields the generator's maximal Gaussian part.

 It is conjectured here that,
 as in the case of orthogonal noises (\cite{LS T-K}),
 the more general Lie--Trotter product formula given in this paper
 has an extension to the class of holomorphic
 quantum stochastic cocycles (\cite{LS holomorphic}).
 By contrast, without orthogonality of noise dimension spaces
 there seems to be no sensible formulation of a Lie--Trotter product formula for quantum stochastic \emph{mapping} cocycles.

 The plan of the paper is as follows.
 In Section~\ref{section: quantum Ito algebra}
some quantum It\^o algebra is developed,
 for studying
 the series product
 on the class of stochastic generators considered here.
 In Section~\ref{section: quantum stochastics}
the relevant
 quantum stochastic analysis is recalled.
 The quantum stochastic Lie--Trotter product formula is proved in
 Section~\ref{section: product formula}.
In Section~\ref{section: gaussian}
the maximal Gaussian component of a quantum stochastic generator is extracted
by means of the concatenation product.
 In the short Section~\ref{section: perturbation},
 the connection to perturbation theory is described,
 and in Section~\ref{section: holomorphic cocycles}
 the quantum It\^o algebra is extended to quadratic form generators and
 a conjecture for quasicontractive holomorphic quantum stochastic cocycles
 is formulated.

\emph{Notation}.
For a vector-valued function $f: \Rplus \to V$
and subinterval $J$ of $\Rplus$,
$f_{\!J}$ denotes the function $\Rplus \to V$
which agrees with $f$ on $J$ and vanishes elsewhere.
For Hilbert spaces $\hil$ and $\hil'$,
$B(\hil; \hil')$ denotes the space of bounded operators from $\hil$ to $\hil'$
and $B(\hil; \hil')_1$ denotes
its closed unit ball, abbreviated to $B(\hil)$ and $B(\hil)_1$ respectively
when $\hil' = \hil$.
For an operator $T \in B(\hil)$,
its real and imaginary parts
are denoted $\re T$ and $\im T$ respectively,
thus $T$ is dissipative if and only if $\re T \les 0$.
The selfadjoint part of a subset $A$ of an involutive space is denoted $A_\sa$.
The predual of $B(\hil)$,
that is the space of ultraweakly continuous linear functionals on $B(\hil)$,
is denoted $B(\hil)_*$.
Algebraic and ultraweak tensor products are denoted by
$\otul$ and $\otol$ respectively and, for vectors $\zeta, \eta \in\hil$,
$\omega_{\zeta, \eta} \in B(\hil)_*$ denotes the functional given by $T \mapsto \ip{\zeta}{T \eta}$.
The symbol $\subset\subset$ is used to denote finite subset.

 \section{Quantum It\^o algebra}
  \label{section: quantum Ito algebra}

 For this section take Hilbert spaces $\hil$ and $\Hil$.
 The block matrix decomposition
 enjoyed by operators in $B(\hil \op \Hil)$
 is frequently appealed to below.
 With respect to the distinguished orthogonal projection
 \[
 \Delta := P_{\{0_\hil\} \op \Hil} =
 \begin{bmatrix} 0_\hil & 0_{\Hil; \hil} \\ 0_{\hil; \Hil} & I_{\Hil} \end{bmatrix}
 \in B(\hil \op \Hil),
 \]
  the composition law on $B(\hil \op \Hil)$ given by
 \[
 F_1 \comp F_2 := F_1 + F_1 \Delta F_2 + F_2,
 \]
with useful alternative asymmetric expressions
\[
F_1 ( I + \Delta F_2 ) + F_2
\ \text{ and } \
F_1 + ( I + F_1 \Delta ) F_2,
\]
 has both
 \begin{subequations}
 \begin{equation}
 \label{eqn: triplea}
 F_1 +F_2 + F_3 +
\big( F_1 \Delta F_2 + F_1 \Delta F_3 + F_2 \Delta F_3  \big) +
 F_1 \Delta F_2 \Delta F_3,
% \ \text{ and }
\end{equation}
and
 \begin{equation}
 \label{eqn: tripleb}
F_1 \comp F_3 +
 (I + F_1 \Delta) \, F_2 \, (\Delta F_3 + I),
\end{equation}
\end{subequations}
 as common expression for
 $F_1 \comp (F_2 \comp F_3)$ and $(F_1 \comp F_2) \comp F_3$.
Moreover the composition $\comp$ has
 $0_{\hil \op \Hil}$ as identity element,
 and the operator adjoint as involution
 since
 $$
(F_1 \comp F_2)^* = F_2^* \comp F_1^*.
$$
%the operator adjoint is an involution with respect to it.
 The notation $\comp$ is taken from the quantum control theory literature, where
 the composition is called the series product.

 Let $\blg := (B(\hil \op \Hil), \comp)$ denote
 the resulting *-monoid (i.e.~involutive semigroup-with-identity),
 let $\beta \in \Real$,
 and consider the following subsets of $\blg$:
\begin{align*}
 \zlg
 &:=
 \big\{
 %\left[ \begin{smallmatrix}
 % K &  \\  & 0
 %\end{smallmatrix} \right] \in \blg:
 K \oplus 0_\Hil:
 K \in B(\hil)
 \big\}
 = \Delta^\perp\, \blg\, \Delta^\perp,
 \\
 \clg
 &:=
 \big\{ F \in \blg: F^* \comp F \les 0 \},
 \\
 \qclgbeta
 &:=
 \clg + \beta \Delta^\perp,
% \quad (\beta \in \Real),
 \\
 \qclg
 &:=
 \clg + \Rplus \Delta^\perp,
 \\
 \ilg
 &:=
 \big\{ F \in \blg: F^* \comp F = 0 \},
 \\
 \ilg^*
 &:=
 \big\{ F^*:  F \in \ilg \},
 \\
 \ulg
 &:=
 \ilg \cap \ilg^*,
\end{align*}
% Also set
% \[
% \xlg :=
% \big\{
% F \in \blg: F =
% \left[ \begin{smallmatrix} K & \\ & 0 \end{smallmatrix} \right],
% \re K \les 0
% \big\}.
% \]
 and for $F \in \qclg$ set
 \begin{equation*}
 \beta_0(F) :=
 \inf \{ \beta \in \Real: F - \beta \Delta^\perp \in \clg \}.
 \end{equation*}
 These classes are relevant to the characterisation of
 the stochastic generators of
quantum stochastic cocycles which are respectively
contractive, quasicontractive
 (with exponential growth bound $\beta_0(F)$),
 isometric, coisometric and unitary
(see Theorem~\ref{thm: cocycle = qsde} below).
Note that,
 for $F \in \qclg$,

 \begin{equation*}
% \label{eqn: beta geq omega}
 F \in \qclgbeta
 \text{ if and only if }
 \beta \ges \beta_0(F).
 \end{equation*}

 \begin{rems}
 Further classes are relevant to the characterisation of
quantum stochastic cocy\-cles which are
 nonnegative, selfadjoint, partially isometric or projection-valued
 (\cite{Wills}).
The characterisation of the generators of
 `pure-noise' (or `local') nonnegative contraction cocycles
 (for which $\hil = \Comp$) plays an important role in
the identification of the
 minimal dilation of a quantum dynamical semigroup
 (\cite{Bhat}).
\end{rems}

 Let $F_1$, $F_2$ and $F$ in $\blg$ have respective block matrix forms
 $\left[ \begin{smallmatrix}
  K_1 & M_1 \\ L_1 & Q_1 - I
 \end{smallmatrix} \right]$,
 $\left[ \begin{smallmatrix}
  K_2 & M_2 \\ L_2 & Q_2 - I
 \end{smallmatrix} \right]$
 and
 $\left[ \begin{smallmatrix}
  K & M \\ L & Q - I
 \end{smallmatrix} \right]$.
 Then
 \begin{align}
 \nonumber
 F_1 \comp F_2  &=
 \begin{bmatrix}
 K_1 + K_2 + M_1 L_2 & M_1 Q_2 + M_2 \\
 L_1 + Q_1 L_2 & Q_1 Q_2 - I
  \end{bmatrix},
 \text{ and }
 \\
  \label{eqn: F star F}
 F^* \comp F  &=
 \begin{bmatrix}
 K^* + K + L^* L & L^* Q + M \\
 M^* + Q^* L & Q^* Q - I
 \end{bmatrix},
 \end{align}
 moreover,
for
 $Z \in \zlg$, $\beta \in \Real$ and $X \in \blg$,
 \begin{align*}
 \nonumber
 (F + Z)^* \comp (F + Z)
 &=
 Z^* +  \, F^* \comp F \, + Z
 \\
 % \label{eqn: minus beta}
 (F - \beta \Delta^\perp)^* \comp (F - \beta \Delta^\perp)
 &=
 F^* \comp F - 2 \beta \Delta^\perp,
 \text{ and }
 \\
 \nonumber
 F^* \comp X \comp F
 &= F^* \comp F +
 (I + \Delta F)^* X (I + \Delta F).
 \end{align*}
 These identities imply the following relations:
 \begin{equation*}
% \label{G in clg}
 (G\comp F)^* \comp (G \comp F)
 \left\{
 \begin{array}{ll}
 =
 F^* \comp F + G^* + G
 & \text{ if }
 G \in \zlg,
 \\
 \les
 F^* \comp F + 2 \beta \Delta^\perp
 & \text{ if }
 G \in \qclgbeta,
 \\
 =
 F^* \comp F
 & \text{ if }
 G \in \ilg.
 \end{array} \right.
 \end{equation*}
 The basic algebraic properties of these subsets of $\blg$
 are collected in the following proposition;
 using the above observations, their proof is straightforward.
 Recall that an operator $T \in B(\hil)$ is \emph{dissipative} if it satisfies
 $\re T \les 0$, that is
 \[
 \re \ip{u}{Tu} \les 0
 \qquad
 (u \in \hil).
 \]
\begin{propn}
 \label{propn: classes}
 In the *-monoid $\blg$ the following hold.
 \begin{alist}
 \item
 Its group of invertible elements is given by
 \[
 \blg^\times =
 \big\{
 \left[ \begin{smallmatrix}
  * & * \\ * & Q
 \end{smallmatrix} \right]
- \Delta
 \in \blg:
 Q \in B( \Hil )^\times
 \big\};
 \]
the identity element being
$0 = \left[ \begin{smallmatrix}
  0 & 0 \\  0 & I
 \end{smallmatrix} \right] - \Delta$, and
 the inverse of
 $\left[ \begin{smallmatrix}
  K & M \\ L & Q
 \end{smallmatrix} \right] - \Delta$ being
 \[
 \begin{bmatrix}
 M Q^{-1} L - K & -M Q^{-1} \\
 - Q^{-1} L & Q^{-1}
 \end{bmatrix}
 - \Delta
 =
 \begin{bmatrix}
 -M  \\
  I
 \end{bmatrix}
 Q^{-1}
 \begin{bmatrix}
 -L &
  I
 \end{bmatrix}
  -
 \begin{bmatrix}
 K  &  0 \\
  0  &   I
 \end{bmatrix}.
 \]
 \item
 Its centre is $ \zlg$.
 \item
Denoting the class of dissipative operators on $\hil$ by $D(\hil)$,
 \begin{align*}
 &
\qclg = \clg + \zlg,
 \text{ and }
 \\
 &
  \clg \cap \zlg =
 \{ T \op 0_\Hil: \, T \in D(\hil) \}.
% \big\{\!
% \left[\begin{smallmatrix} K & \\
% & 0 \end{smallmatrix}\right]:
% \re K \les 0 \big\}$.
 \end{align*}
 \item
 $\qclg$, $\clg$, $\ilg$ and $\ulg$
 are submonoids\tu{;}
 their groups of invertible elements are given by
 \begin{align*}
 &\ilg^\times =
 \clg^\times = \ulg, \text{ and }
 \\
 &\qclg^\times = \ulg + \zlg.
 \end{align*}
 Moreover,
 \[
  \ilg \cap \zlg = \ulg \cap \zlg = i\, \zlg_{\sa}.
 \]
 \end{alist}
\end{propn}
\begin{rem}
Clearly $\ulg$ is also closed under taking adjoints,
 and so is a sub-\!*-monoid of $\blg$.
\end{rem}

 In fact $\qclg$ and $\clg$ are sub-\!*-monoids too;
 this is not immediately obvious.
 It follows from Part (e) of Theorem~\ref{thm: LW, GLSW} below,
 but there are now also direct proofs.
 The one given below is based on
 an elegant argument of Wills,
 arising as a biproduct of his analysis of partially isometric quantum stochastic cocycles
 (\cite{Wills}).

  \begin{propn}
  \label{propn: wills id}
Let $F \in \blg$ and $T \in \zlg_\sa$.
Then
\begin{equation}
\label{eqn: wills id}
( \Delta F + I )^* ( F \comp F^* ) ( \Delta F + I )
=
F^* \comp F + ( F^* \comp F ) \, \Delta \, ( F^* \comp F ),
\end{equation}
and
$$
F^* \comp F \les T
\text{ if and only if }
F \comp F^* \les T.
$$
\end{propn}

\begin{proof}
By the associativity of $\comp$,
setting $F_1 = F^*$, $F_2 = F \comp F^*$  and $F_3 = F$
in the expression~\eqref{eqn: tripleb}
for $F_1 \comp F_2 \comp F_3$ we see that
\begin{align*}
F^* \comp F +
\LHS(\ref{eqn: wills id})
&
=
F^* \comp ( F \comp F^* ) \comp F
\\
&
=
( F^* \comp F ) \comp ( F^* \comp F )
=
\RHS(\ref{eqn: wills id}) + F^* \comp F,
\end{align*}
so identity~\eqref{eqn: wills id} holds.
Suppose now that
$F \comp F^* \les T$.
Then, since $\Delta T = 0 = T \Delta$,
$\LHS(\ref{eqn: wills id}) \les T$
and so,
since
 $\Delta \ges 0$ and
 $F^* \comp F$ is selfadjoint,
$$
F^* \comp F
=
 \eqref{eqn: wills id} - ( F^* \comp F ) \Delta ( F^* \comp F )
\les \eqref{eqn: wills id}
\les T.
$$
The converse implication follows by exchanging $F$ and $F^*$.
\end{proof}

% \end{document}
 Consider the following possible block matrix forms for
 $F \in \blg$:
 %experiment (failed)
% \begin{align*}
% &\ulg =
% \Big\{
% F = \left[ \begin{smallmatrix}
%  K & M \\ R & W - I
% \end{smallmatrix} \right] \in \blg:
% \\
% &\qquad \qquad \qquad \qquad
% K = iH - \frac{1}{2} R^*R \text{ and } M =  -R^*W,
% \\
% &\qquad \qquad \qquad \qquad
% \text{ where } H=H^* \text{ and } W \text{ unitary}
% \Big\},
% \\
% &\clg =
% \Big\{
% F = \left[ \begin{smallmatrix}
%  K & M \\ L & C - I
% \end{smallmatrix} \right] \in \blg:
% \\
% &\qquad \qquad \qquad \qquad
% K = iH - \frac{1}{2} (L^*L + B^2)
% \text{ and }
% M -L^*C -B V (I-C^*C)^{1/2},
% \\
% &\qquad \qquad \qquad \qquad
% \text{ where }
% H=H^*,\, B=B^*\ges 0, \text{ and } \norm{C}, \norm{V} \les 1
% \Big\};
% \end{align*}
%
\begin{subequations}
% \label{eqn: star}
 \begin{align}
  \label{star}
 % &\qquad\qquad
 &\begin{bmatrix}
\beta I +  iH  - \frac{1}{2} (L^*L + A^2) & & -L^*C -A D (I - C^*C)^{1/2}
\\
% \\
% & &
% \\
L & & C - I
 \end{bmatrix},
\ \text{ and}
 \\
\nonumber
\\
 % \nonumber
 %  &\!\!\!\!\!\!\!\!\!\!\text{ where } S = (I - C^*C)^{1/2},
 %  \text{and }
 % \\
 %\nonumber
 %&(\beta, H, L, C, A, D) \in
 %\Real \times B(\hil)_{\sa}
 %\times B(\hil; \Hil)
 %\times C(\Hil)
 %\times B(\hil)_+
 %\times C(\Hil; \hil);
 % \end{align}
 %\begin{align}
 \label{**}
 %&\qquad\qquad F =
 &\begin{bmatrix}
 \beta I +  iH - \frac{1}{2} ( MM^* + B^2 ) & & M
\\
%\\
  - CM^* - (I - C C^*)^{1/2} E^*B  & & C - I
 \end{bmatrix},
% \\
%  \nonumber
% &\!\!\!\!\!\!\!\!\!\!
% \text{ where } T = (I - C C^*)^{1/2},
% \text{ and }
% \\
% \nonumber
%  &(\beta, H, M, C, B, E) \in
% \Real \times B(\hil)_{\sa}
% \times B(\Hil; \hil)
% \times C(\Hil)
% \times C(\hil)_+
% \times C(\hil; \Hil);
% \end{align}
% and
% \begin{equation}
% \label{***}
% % F =
% &\begin{bmatrix}
%  iH - \frac{1}{2}L^*L & -L^*C \\ L & C - I
% \end{bmatrix}
% \\
% \label{****}
% % F =
% &\begin{bmatrix}
%  iH - \frac{1}{2}M M^* & M \\ - CM^* & C - I
% \end{bmatrix}
 \end{align}
\end{subequations}
 in which $\beta \in \Real$, $H \in B(\hil)_{\sa}$, $A,B \in B(\hil)_+$,
$C \in B(\Hil)_1$ and $D,E \in B(\Hil;\hil)_1$,
so that $H = \im K$.

%     \begin{comment} %%%%%%%%%%%%%%%%%%%%%%%%% COMMENTED OUT

%         \end{comment}  %%%%%%%%%%%%%%%%%%%%%%%%% COMMENTED OUT

 \begin{thm}[\emph{Cf.} \cite{LWqsde}, \cite{GLSW}]
 \label{thm: LW, GLSW}
 Let $\beta \in \Real$ and $F \in \blg$.
 \begin{alist}
 \item
 The following are equivalent\tu{:}
 \begin{rlist}
 \item
 $F \in \qclgbeta$.
 \item
 $F$ has block matrix form~\eqref{star}.
\item
 $F$ has block matrix form~\eqref{**}.
 \item
 $F^* \in \qclgbeta$.
 \end{rlist}
 \item
 $F \in \ilg$ if and only if
 $F$ has block matrix form~\eqref{star},
 with $\beta = 0$, $A = 0$ and $C$ isometric.
 \item
 $F \in \ilg^*$ if and only if
 $F$ has block matrix form~\eqref{**},
 with $\beta = 0$, $B = 0$ and $C$ coisometric.
 \end{alist}

 \begin{itemize}
 \item[(d)]
 Let $\Hil_1 \op \Hil_2$ be an orthogonal decomposition of $\Hil$.
 Then,
with respect to the inclusion
$J_1: \hil \op \Hil_1 \to \hil \op \Hil$,
 \begin{equation}
 \label{eqn: Jone}
J_1^* \qclgbeta J_1 = \qclgonebeta.
\end{equation}

 \item[(e)]
 Set
  $\Hil' := \Hil \op (\Hil \op \hil)$.
Then, in terms of the inclusion
$J: \hil \op \Hil \to
( \hil \op \Hil ) \op ( \Hil \op \hil )
=
\hil \op \Hil'$,
$$
\clg = J^* \ulgprime J.
$$
 \end{itemize}
 \end{thm}

\begin{rem}
In block matrix form
 \begin{align*}
&
 J_1 =
  I_{\hil} \oplus
  \begin{bmatrix}
 I_{\Hil_1}   \\
   0_{\Hil_1; \Hil_2}
 \end{bmatrix}
 \in B(\hil \op \Hil_1; \hil \op \Hil),
\ \text{ and } \
%\\
% &
 J =
  I_{\hil} \oplus
  \begin{bmatrix}
 I_{\Hil}   \\
   0_{\Hil; \Hil \op \hil}
 \end{bmatrix}
 \in B(\hil \op \Hil; \hil \op \Hil').
\end{align*}
\end{rem}

\begin{proof}
 The proof exploits the fact that, for an operator
 $T \in B(\hil \op \Hil)$, $T \ges 0$ if and only if
 $T$ has block matrix form
 $\left[ \begin{smallmatrix}
 X & X^{1/2} V Z^{1/2} \\ Z^{1/2} V^* X^{1/2} & Z
 \end{smallmatrix} \right]$ where
 $(X, Z, V) \in B(\hil)_+ \times B(\Hil)_+ \times B(\Hil; \hil)_1$.

(a)
 Set $G = F - \beta \Delta^\perp$.
 % \\
 Suppose first that (ii) holds. Then $G$ is given by~\eqref{star} with $\beta = 0$
 and,
setting $S := (I - C^* C)^{1/2}$,
 \[
 - G^* \comp G =
 \begin{bmatrix}
 A^2 & ADS \\ SD^*A & S^2
 \end{bmatrix}
 \ges 0,
 \]
 so $G \in \clg$. Thus (ii) implies (i).
 Conversely, suppose that (i) holds and let
 $\left[ \begin{smallmatrix}
 K & M \\ L & C-I
 \end{smallmatrix} \right]$
 be the block matrix decomposition of $G$.
 Then $G \in \clg$ so
 \[
 \begin{bmatrix}
 K^* + K + L^*L & M + L^*C \\ M^* + C^*L & C^*C - I
 \end{bmatrix}
 =  G^* \comp G
 \les 0.
 \]
 Thus $(K^* + K + L^*L) \les 0$, $\|C\| \les 1$ and,
  for some contraction operator $D$,
\[
M =
-L^* C -
[ - (K^* + K + L^*L) ]^{1/2}
D
(I - C^* C)^{1/2}
\]
 so $G$ is given by~\eqref{star} with $\beta = 0$ and $H = \im K$.
 Thus (i) implies (ii).

 Therefore (i) and (ii) are equivalent;
 taking adjoints we see that (iv) and (iii) are equivalent too.
 The equivalence of (i) and (iv) follows from
 Proposition~\ref{propn: wills id}.

 (b)
This is an immediate consequence of~\eqref{eqn: F star F}.

(c)
 This follows from Part (b), by taking adjoints.

(d)
For $F_1 \in \qclgonebeta$,
$$
F_1 = J_1^* F J_1,
\ \text{ where } \
F = \begin{bmatrix}
F_1 & 0 \\ 0 & 0
\end{bmatrix}
\in
\qclgbeta,
$$
so $\RHS \subset \LHS$ in~\eqref{eqn: Jone}.
For the reverse inclusion, setting
 $\Delta_1 := 0_\hil \op I_{\Hil_1}$ and
 $\Delta := 0_\hil \op I_{\Hil}$,
 \[
 \Delta - J_1 \Delta_1 J_1^* =
 0_\hil \op I_{\Hil_1} \op I_{\Hil_2} -
 0_\hil \op I_{\Hil_1} \op 0_{\Hil_2} =
 0_\hil \op 0_{\Hil_1} \op I_{\Hil_2}
 \ges 0.
 \]
 Thus, for $F \in \clg$, the operator
 $F_1:= J_1^* F J_1$ is in $\clgone$ since
 \begin{align*}
 F_1^* \comp F_1
 &=
 J_1^* F^* J_1 +
 J_1^* F J_1 +
 J_1^* F^* J_1 \Delta_1 J_1^* F J_1
\les
J_1^* ( F^* \comp F ) J_1 \les 0.
 \end{align*}
% and so $F_1 \in \clgone$.
Therefore
$J_1^* \clg J_1 \subset \clgone$
and so, since
% Since
 $J_1^* \beta \Delta^\perp J_1 = \beta \Delta_1^\perp$
 ($\beta \in \Real$),
 $\LHS \subset \RHS$.
% and(d) follows.

 (e)
In view of (d),
it suffices to show that
$J^*  \ulgprime J \supset \clg$.
Accordingly, let
 $F \in \clg$.
 Then, by what we have proved already,
 $F$ has block matrix form~\eqref{star} with
 $\beta = 0$.
 Setting
 \[
 F' :=
 \begin{bmatrix}
 iH - \frac{1}{2}(L^*L + A^2) & -(L^*C + ADS) & L^*T - ADC^* & -AR
 \\
 L & C-I & -T & 0
 \\
 D^*A & S & C^* - I & 0
 \\
 RA & 0 & 0 & 0
 \end{bmatrix},
 \]
 where $R = (I - DD^*)^{1/2}$, $S = (I - C^* C)^{1/2}$, and $T = (I - C C^*)^{1/2}$,
 it is now easily verified that $F' \in \ulgprime$.
 Since $J^* F' J = F$, (e) follows
 and the proof is complete.
 \end{proof}

 \begin{rems}
 The dilation property (e)
 is effectively proved in~\cite{GLSW},
 under the assumption that $\hil$ and $\Hil$ are separable.
If $\Hil = \hil \ot \noise$ for a Hilbert space $\noise$
(as it is in the application to QS analysis,
where $\noise$ is the noise dimension space)
then, in (e),
$\Hil' = \hil \ot \noise'$ where $\noise' = \noise \op \noise \op \Comp$.

In terms of its  block matrix form
$\left[ \begin{smallmatrix} K & M \\ L & C - I \end{smallmatrix} \right]$,
and bound $\beta_0 := \beta_0(F)$,
the equivalent conditions for $F \in \blg$ to be in $\qclg$ read respectively as follows:

\begin{subequations}
% \label{eqn: ineq}
\begin{align*}
% \label{eqn: ineq a}
&
\begin{bmatrix}
2( \re K - \beta_0 I_\hil) + L^*L
&
M + L^* C
\\
M^* + C^* L
&
C^*C - I_\Hil
\end{bmatrix}
\les 0,
\text{ and }
\\
% \label{eqn: ineq b}
&
\begin{bmatrix}
2( \re K - \beta_0 I_\hil) + M M^*
&
L^* + M C^*
\\
L + C M^*
&
C C^* - I_\Hil
\end{bmatrix}
\les 0.
\end{align*}
\end{subequations}
  \end{rems}

%%%%%%%%%%%%%%%%%
%%%%%%%%%%%%%%%%%%%%%%%

\begin{propn}[Left and right series decomposition]
\label{propn: canon series}
 Let $F =
 \left[ \begin{smallmatrix} K & M \\ L & C - I
 \end{smallmatrix} \right]
 \in \qclg$
and set $ \beta_0 =  \beta_0 (F)$.
 Then, setting
\begin{align*}
&
F_1 =
 \begin{bmatrix}
\beta_0 I_\hil + i \im K & 0
 \\
  0 & 0
 \end{bmatrix},
\
F_2^{\ell} =
 \begin{bmatrix}
 - \frac{1}{2} L^*L & -L^* \\ L & 0
 \end{bmatrix},
\
 F_3^{r} =
 \begin{bmatrix}
 - \frac{1}{2} MM^* & M
 \\
  -M^* & 0
 \end{bmatrix},
\\
&
  F_3^{\ell} :=
 \begin{bmatrix}
\re K +  \frac{1}{2} L^*L - \beta_0  I_\hil
& M + L^* C
\\
0
&
C - I
  \end{bmatrix},
\ \text{ and } \
F_2^{r} :=
\begin{bmatrix}
\re K +  \frac{1}{2} M M^* - \beta_0  I_\hil
&
0
\\
L + C M^*
&
C - I
  \end{bmatrix},
\end{align*}
the following hold\tu{:}
$$
F_1 \in  \zlg,
\ \
F_2^{\ell},  F_3^{r} \in \ulg,
\ \
F_2^{r},  F_3^{\ell}, \in \clg,
$$
and
 \[
 F_1 \comp F_2^{\ell} \comp F_3^{\ell}
 =
F
=
 F_1 \comp F_2^{r} \comp F_3^{r},
 \]
 \end{propn}

 \begin{proof}
It follows from Theorem~\ref{thm: LW, GLSW}, and the above block matrix inequalities characterising membership of $\qclg$, that
$$
F_2^\ell, F_3^{r} \in \ulg
\ \text{ and } \
F_3^\ell, F_2^{r} \in \clg.
$$
The two series decompositions follow from the identities
 \[
 F_1 \Delta  F_2^{\ell} = 0, \
  F_2^{\ell} \Delta  F_3^{\ell} =
 \begin{bmatrix}
 0 & L^* - L^* C \\ 0 & 0
 \end{bmatrix}, \
 F_1 \Delta  F_2^{r} = 0
  \ \text{ and } \
  F_2^{r} \Delta  F_3^{r} =
 \begin{bmatrix}
 0 & 0 \\  M^* - C M^* & 0
 \end{bmatrix}.
 \]
The rest is clear.
\end{proof}

\begin{rems}
(i)
The left and right series decompositions are related
via the adjoint operation as follows:
\[
 (F^*)_1  =  (F_1)^*,
\
 (F^*)_2^{\ell} = (F_3^{r})^*
\ \text{ and } \
 (F^*)_3^{\ell} = (F_2^{r})^*.
 \]

(ii)
Let
 $F =
 \left[ \begin{smallmatrix} K & M \\ L & C - I
 \end{smallmatrix} \right]
 \in \qclg$
and set
$H := \im K$.
Then
$F \in \ilg$ if and only if
$F$ has left series decomposition
\[
i
 \begin{bmatrix}
 H  & 0
 \\
  0 & 0
 \end{bmatrix}
 \comp
 \begin{bmatrix}
 - \frac{1}{2} L^*L & -L^* \\ L & 0
 \end{bmatrix}
 \comp
 \begin{bmatrix}
 0 &  0 \\ 0 & C-I
  \end{bmatrix},
\]
 with $C$ isometric,
whereas
$F \in \ilg^*$ if and only if
$F$ has right series decomposition
\[
i
 \begin{bmatrix}
 H  & 0
 \\
  0 & 0
 \end{bmatrix}
 \comp
 \begin{bmatrix}
 0 &  0 \\ 0 & C-I
  \end{bmatrix}
 \comp
 \begin{bmatrix}
 - \frac{1}{2} M M^* &M \\ -M^* & 0
 \end{bmatrix},
\]
 with $C$ coisometric.
\end{rems}

%%%%%%%%%%%%%%%%%%%%%%%%%%%%%%
 Now suppose that
 $\Hil$ has an orthogonal decomposition $\Hil_1 \op \Hil_2$.
 Define injections
  \begin{align}
\nonumber
 &
 \iota:
\blgone \to \blg,
\quad
 F = \begin{bmatrix}
 K & M \\ L & N \end{bmatrix}
 \mapsto
 F \op 0_{\Hil_2} =
 \begin{bmatrix}
 K & M & 0 \\
 L & N & 0 \\
 0 & 0 & 0  \end{bmatrix}, \text{ and}
 \\
 &
\label{eqn: iota con}
 \iota':
\blgtwo \to \blg,
\quad
 F = \begin{bmatrix}
 K & M \\ L & N \end{bmatrix}
 \mapsto
 \Sigma \big( F \op 0_{\Hil_1}  \big) =
 \begin{bmatrix}
 K & 0 & M \\
 0 & 0  & 0 \\
 L & 0 & N
  \end{bmatrix},
\end{align}
 $\Sigma$ being the sum-flip map
 $B(\hil \op \Hil_2 \op \Hil_1) \to B(\hil \op \Hil)$,
 and define the composition
 \[
 \blgone \times \blgtwo \to \blg, \quad
 (F_1, F_2) \mapsto
F_1 \boxplus F_2 := \iota (F_1) + \iota'(F_2),
 \]
%  which is
 known as the \emph{concatenation product} in quantum control
 theory (\cite{GoughJames}).
 Thus
 \begin{equation*}
% \label{eqn: concat}
\begin{bmatrix}
 K_1 & M_1  \\
 L_1 & N_1
 \end{bmatrix}
 \boxplus
 \begin{bmatrix}
 K_2 & M_2  \\
 L_2 & N_2
 \end{bmatrix}
  =
 \begin{bmatrix}
 K_1 + K_2 & M_1 & M_2 \\
 L_1 & N_1 & 0 \\
 L_2 & 0 & N_2  \end{bmatrix}.
 \end{equation*}
 Note that
 \[
 \iota (F_1) = F_1 \boxplus 0_{\Hil_2}
\text{ and }
 \iota'(F_2) =   0_{\Hil_1}  \boxplus F_2.
\]
In view of the identity
\[
 \iota(F_1) \, \Delta \, \iota'(F_2)
 = 0_{\hil \op \Hil}
 =
 \iota'(F_2) \, \Delta \, \iota(F_1),
\]
for
$F_1 \in \blgone$ and $F_2 \in \blgtwo$,
 the concatenation product is effectively a special case
 of the series product:
 \begin{equation}
  \label{eqn: effectively}
 F_1 \boxplus F_2 =
 \iota( F_1) \, \comp \, \iota'(F_2 ).
 \end{equation}

 We end this section with a significant representation of the quantum It\^o algebra.
 It is relevant to
 the realisation of QS cocycles as time-ordered exponentials (\cite{Holevo}),
 and also to
 the convergence of a class of scaled quantum random walks to QS cocycles (\cite{BGL}).
 Set
 $\mathcal{S} := I_{ \hil \op \Hil \op \hil } + \Al $,
 where
 $ \Al$ is the following subalgebra of $B( \hil \op \Hil \op \hil )$:
 \[
 \big\{
 T \in B( \hil \op \Hil \op \hil ): P_{ \hil \op \Hil \op \{ 0 \} } T = T = T P_{ \{ 0 \} \op \Hil \op \hil }
 \big\}\tu{;}
 \]
 thus $\mathcal{S}$ consists of the elements of $B( \hil \op \Hil \op \hil )$ having
 block matrix form
 $\left[ \begin{smallmatrix} I_\hil & * & * \\ 0 & * & * \\ 0 & 0 & I_\hil \end{smallmatrix} \right]$.
 With respect to
(operator composition and)
the involution given by
 \[
 T \mapsto T^\star
 := \Xi T^* \Xi,
 \ \text{ where } \
 \Xi :=
 \begin{bmatrix}
 0 & 0 & I_\hil \\ 0 & I_\Hil & 0 \\ I_\hil & 0 & 0
 \end{bmatrix},
 \]
 $\mathcal{S}$ is
 a sub-*-monoid of $B( \hil \op \Hil \op \hil )$,
 and
 the following is readily verified.

 \begin{propn}[\cite{Holevo},~\cite{Belavkin}]
 \label{propn: monoidiso}
 The prescription
 \[
  \begin{bmatrix}
 K & M \\ L & Q-I
 \end{bmatrix}
 \mapsto
  \begin{bmatrix}
 I_\hil & M & K \\ 0 & Q & L \\ 0 & 0 & I_\hil
 \end{bmatrix}
 =
 I_{ \hil \op \Hil \op \hil } +
 \begin{bmatrix}
 0 & M & K \\ 0 & Q-I & L \\ 0 & 0 & 0
 \end{bmatrix}
 \]
 defines an isomorphism of *-monoids
$\phi:
 \blg \to \mathcal{S}$.
 \end{propn}

 \section{Quantum stochastics}
 \label{section: quantum stochastics}

\emph{For the rest of the paper fix Hilbert spaces
 $\init$ and $\noise$, and set $\khat := \Comp \op \noise$
and}
$\Kil := L^2(\Rplus; \noise)$.
 The quantum It\^o algebra developed in the previous section
 is applied below with $\hil = \init$ and $\Hil = \noise \ot \init$,
 so that
 $\hil \op \Hil = \khat \ot \init$.
 In  this context, the operator
 \[
 \Delta =  \Delta_{\noise} := P_{\{0_{\Comp}\} \op \noise} =
 \begin{bmatrix} 0 & 0 \\ 0 & I_{\noise} \end{bmatrix}
 \in B(\khat)
 \]
 is ubiquitous;
 below it is freely ampliated.

 In this section we collect the quantum stochastic (QS) facts needed below.
 For more detail, see~\cite{L}; for further background, see~\cite{Partha} and~\cite{Meyer}.
 Let $\Fock$ denote the symmetric Fock space over $\Kil$.
  We use normalised exponential vectors
 \[
 \vp(g) :=
 e^{-\norm{g}^2/2} \ve(g)
 \text{ where }
 \ve(g) :=
 \big( (n!)^{-1/2} g^{\ot n} \big)_{n\ges 0}
 \qquad
 (g \in \Kil),
 \]
 in terms of which the \emph{Fock--Weyl operators} are the unitary operators on $\Fock$
 determined by the identity
 \[
 W(f) \vp(g) = e^{-i \im \ip{f}{g}} \vp(f+g)
 \qquad
 (f,g\in \Kil),
 \]
 and the \emph{second quantisation operators} are the contractions from
 $\Fock$ to $\Fock'$
 determined by the identity
 \begin{equation}
 \label{eqn: second quant}
 \Gamma(C) \ve(g) := \ve(Cg)
 \qquad
 (C \in B(\Kil; \Kil')_1, \, g \in \Kil),
 \end{equation}
where
$\Fock'$ is the symmetric Fock space over a Hilbert space $\Kil'$.
Note that
$\Gamma(C)$ is isometric if $C$ is, and $\Gamma(C^*) = \Gamma(C)^*$.
 There is a useful family of slice maps:
 \[
 \Omega(g',g) :=
 \id_{B(\hil; \hil')} \otol \omega_{\vp(g'), \vp(g)}:
 B(\hil; \hil') \otol B(\Fock) \to B(\hil; \hil')
 \qquad
 (g',g\in \Kil),
 \]
 amongst which
 $\Expect := \Omega(0,0)$
 is referred to as the \emph{vacuum expectation}.
 (The Hilbert spaces $\hil$ and $\hil'$ are determined by
 context.)
 Moreover, the action of the
 slice maps extends to the class of unbounded operators from
 $\hil \ot \Fock$ to $\hil' \ot \Fock$ whose domains include $\hil \otul \Exps$,
 where $\Exps:= \Lin \{ \vp(g): g \in \Kil \}$.
 The following obvious identity is exploited below:
 \begin{equation}
  \label{eqn: slices}
\Omega(g',g)(T) =
 \Expect \big[
 (I_{\hil'} \ot W(g'))^*\, T\, (I_{\hil} \ot W(g))
 \big].
 \end{equation}
We also need two families of endomorphisms of (the von Neumann algebra) $B(\Fock)$:
\[
\rho_t: T \mapsto R_t T R_t
\ \text{ and } \
\sigma_t: T \mapsto I_{\Fock_{[0,t[}} \ot S_t T S_t^*
\qquad
(t \in \Rplus),
\]
where
$R_t := \Gamma ( r_t )$
and
$S_t := \Gamma ( s_t )$
for the unitary operators
$r_t$ and $s_t$ defined as follows
\begin{align*}
&
r_t: \Kil \to \Kil,
\quad
(r_t f)(s) =
 \begin{cases}
f(t-s)
& \text{if } s \in [0,t[
\\
f(s) & \text{if } s \in [t, \infty[
\end{cases},
\\
&
s_t: \Kil \to \Kil_{[t, \infty [},
\quad
(s_t f)(s) =
f(s-t),
\end{align*}
where $ \Kil_{[t, \infty [} := L^2 ( [t, \infty[ ; \noise )$.
The \emph{time-reversal maps} $\big( \rho_t \big)_{t\ges 0}$
are obviously involutive:
$\rho_t^2 = \id_{B(\Fock)}$,
and the {time-shift maps} $(\sigma_t)_{t \in \Rplus}$ form a semigroup,
known as the
\emph{CCR flow of index} $\noise$.
Both are freely ampliated to act on $B(\hil; \hil') \otol B(\Fock)$,
for Hilbert spaces $\hil$ and $\hil'$.
For use below,
note the identity
\begin{equation}
 \label{eqn: Omega sigma}
\Omega\big( g'_{[r, \infty[}, g_{[r, \infty[} \big) \circ \sigma_r
=
\Omega\big(s_r^* (g'|_{[r, \infty[}), s_r^*(g|_{[r, \infty[}) \big)
\qquad
(g, g' \in \Kil, r \in \Rplus).
\end{equation}

Let
 $X^i = \big( X^i_0 + \Lambda_t(H^i) \big)_{t\ges 0}$  for a
 QS integrand $(\hil_i,\hil_{i-1})$-process $H^i$
 and
 bounded operator $X^i_0 \in B(\hil_i;\hil_{i-1}) \ot I_{\Fock}$,
where $i=1,2$.
 Suppose that all of the processes
 $H^1, H^2, X^1$ and $X^2$ are bounded,
 that is they consist of bounded operators,
 and $H^1$ and $H^2$ are strongly continuous.
 Then the quantum It\^o product formula (\cite{HuP}) reads
 \begin{equation}
 \label{eqn: QI}
 X^1_t X^2_t =
 X^1_0 X^2_0 + \Lambda_t (H)
  \text{ where }
 H =
\big(
 H^1_s (I_{\khat} \ot X^2_s) +
 (I_{\khat} \ot X^1_s) H^2_s +
 H^1_s \Delta H^2_s
\big)_{s \ges 0},
 \end{equation}
 with $\Delta$ abbreviating $\Delta \ot I_{\hil_1 \ot \Fock}$.

 Let $X = \big( \Lambda_t(H) \big)_{t\ges 0}$
 for a bounded QS integrand $(\hil_1,\hil_2)$-process $H$
 with block matrix form
 $\left[\begin{smallmatrix} K & * \\ * & * \end{smallmatrix}\right]$,
 then
 \begin{equation}
  \label{eqn: expect}
 \Expect [ X_t ]\, v =
 \int_0^t \ud s\
 \Expect [K_s]\, \!v
 \qquad
 (v \in \hil_1, t \in \Rplus),
\end{equation}
 and,
 if the process $X$ is bounded then,
   for bounded operators
 $R\in B(\hil_2; \hil_3) \ot I_\Fock$ and
 $S\in B(\hil_0; \hil_1) \ot I_\Fock$,
 \begin{equation}
  \label{eqn: under}
 R \Lambda_t (H) S =
 \Lambda_t \big(
 (I_{\khat} \ot R) H_\cdot (I_{\khat} \ot S)
 \big)
\qquad
(t \in \Rplus).
 \end{equation}

 A bounded  \emph{QS} (\emph{left}) \emph{cocycle} on $\init$,
 with noise dimension space $\noise$,
 is a bounded process $V$ on $\init$
 which satisfies
 \[
V_0 = I_{\init \ot \Fock}
\ \text{ and } \
 V_{s+t} = V_s \sigma_s(V_t)
 \qquad
 (s,t\in\Rplus).
 \]
 It is a \emph{QS right cocycle} if instead
 it satisfies $V_{s+t} = \sigma_s(V_t) V_s$
 for $s,t\in\Rplus$.
 If $V$ is a QS left cocycle then,
  $V^* := (V^*_t)_{t\ges 0}$ and $\reversedco := \big( \rho_t ( V_t ) \big)_{t\ges 0}$ define
QS right cocycles, and
 \begin{equation}
  \label{eqn: dual}
\dualco := \big(  \rho_t ( V^*_t  ) = \rho( V_t )^*  \big)_{t\ges 0}
 \end{equation}
defines a QS left cocycle,
 called the \emph{dual cocycle} of $V$ (\cite{Journe}).
This said, \emph{we work exclusively with QS left cocycles in this paper}.
 Following standard terminology of semigroup theory (\cite{HiP}),
 \[
 \beta_0(V) :=
 \inf \{ \beta \in \Real:
 \sup\nolimits_{t>0} \| e^{-\beta t} V_t \| < \infty
 \}
 \]
 is referred to as the \emph{exponential growth bound} of $V$.
 If $V$ is strongly continuous then $\beta_0 (V) < \infty$;
 $V$ is called \emph{quasicontractive} if,
 for some $\beta \in \Real$,
 the QS cocycle $( e^{-\beta t} V_t )_{t\ges 0}$ is contractive.
 When $V$ is locally uniformly bounded,
it is called \emph{elementary} (or \emph{Markov regular}) if
  its expectation semigroup
 $\big( \Expect [ V_t ] \big)_{t\ges 0}$ is norm continuous.

 Note that, given a QS cocycle $V$, the two-parameter family
 $V_{r,t} := \sigma_r(V_{t-r})$ satisfies
 $V_{0,t} = V_t$ ($t \in \Rplus$),
 the evolution equation
% \begin{subequations}
% \label{eqn: QSE}
 \begin{equation*}
% \label{subeqn: QSEa}
V_{r,r} = I_{\init \ot \Fock}
\ \text{ and } \
 V_{r,t} = V_{r,s} V_{s,t}
 \qquad
 (t \ges s \ges r),
 \end{equation*}
 and the biadaptedness property
\begin{equation*}
% \label{subeqn: QSEb}
 V_{r,t} \in
 \big( B(\init) \ot I_{[0,r[} \big) \otol \big( B(\Fock_{[r,t[}) \ot I_{[t,\infty[} \big)
 \qquad
 (t \ges r \ges 0).
 \end{equation*}
 For use below, note that in the notation~\eqref{eqn: slices},
these \emph{QS evolutions} satisfy
the time-covariance identity
 \begin{subequations}
 \label{eqn: QSE}
 \begin{equation}
\label{subeqn: QSEc}
 \Omega(  c'_{[r,t[}, c_{[r,t[} ) (V_{r,t})
 =
 \Omega( c'_{[0,t-r[}, c_{[0,t-r[} ) (V_{t-r})
\end{equation}
by~\eqref{eqn: Omega sigma},
and
further evolution identity
\begin{equation}
\label{subeqn: QSEd}
 \Omega(  g'_{[r,t[}, g_{[r,t[} ) (V_{r,t})
 =
 \Omega( g'_{[r,s[}, g_{[r,s[} ) (V_{r,s})
 \Omega( g'_{[s,t[}, g_{[s,t[} ) (V_{s,t})
 \end{equation}
 \end{subequations}
 for $c', c \in \noise$, $g', g \in \Ltwoloc(\Rplus;\noise)$ and
 $t \ges s \ges r \ges 0$.

 \begin{rem}
 Note that the relations~\eqref{subeqn: QSEc} and~\eqref{subeqn: QSEd} also hold
 when $V$ is a product of such QS evolutions.
 \end{rem}

 \begin{thm}[\cite{Fagnola}, [$\text{LW}\!_{1,2}$]]
 \label{thm: cocycle = qsde}
 For $F \in \Qqclgbeta$ where $\beta \in \Real$,
 the \tu{(}left\tu{)}
 QS differential equation
 $\ud X_t = X_t \cdot F \ud \Lambda_t$, $X_0 = I$ has a unique weakly regular weak solution,
 denoted $X^F$,
 moreover $X^F$ is a quasicontractive, elementary QS cocycle,
 with exponential growth bound at most $\beta$,
 strongly satisfying its QS differential equation.
 The resulting map
% \[
% \Qqclg \to \QSqcCMr,
$F \mapsto X^F$
% \]
 is bijective
 from $\Qqclg$ to
 the class of quasicontractive elementary QS cocycles\tu{;} it
 restricts to bijections from
 $\Qclg$, $\Qilg$ and $\Qilg^*$
 to the respective subclasses of
 contractive, isometric and coisometric, elementary QS
 cocycles, moreover it satisfies
 $X^{F^*} =  (X^F)^{\dual}$.
 \end{thm}

 The unique operator $F$ associated with a
quasicontractive elementary QS cocycle $V$ in this way
 is referred to as the \emph{stochastic generator} of $V$.

There is a basic class of QS cocycles which plays an important role.

\begin{eg}
\label{examp: Weyl}
For each $c \in \noise$, the (Fock) \emph{Weyl cocycle} $W^c$ is given by
 \begin{equation*}
% \label{eqn: Weyl}
 W^c := \big( I_\init \ot W(c_{[0,t[}) \big)_{t\ges 0}
 \qquad
 (c \in \noise).
 \end{equation*}
 These are QS unitary cocycles (both left and right),
and $X = W^c$ satisfies
 \begin{equation*}
%  \label{*}
X^{\reversed} = X
\ \text{ and } \
X^{\dual} = W^{-c} = X^*,
 \end{equation*}
 moreover $X$ is elementary with stochastic generator
 \begin{equation*}
%  \label{eqn: F(c)}
 F_c :=
  \begin{bmatrix}
 -\frac{1}{2} \norm{c}^2 & -\bra{c}  \\
 \ket{c} & 0
 \end{bmatrix} \ot I_\init
 \in B(\khat\ot\init).
 \end{equation*}
\end{eg}

 \begin{propn}
 \label{propn: VW}
 Let $F \in \Qqclg$ and $G \in \QqclgC \ot I_\init \subset \Qqclg$. Then
 \[
 X^F X^G = X^{F \comp G}
 \text{ and }
 X^G\, X^F = X^{G \comp F}.
 \]
 \end{propn}
\begin{proof}
 First note that, since $X^G$ is of the form
 $\big( I_\init \ot V_t \big)_{t\ges 0}$
 for a quasicontractive QS cocycle $V$ on $\Comp$,
 $$
( F \ot I_\Fock ) ( I_{\khat} \ot X^G_s )
=
F \ot V_s
=
 ( I_{\khat} \ot X^G_s ) ( F \ot I_\Fock )
\qquad (s \in \Rplus).
$$
 Therefore, by
 the quantum It\^o product formula~\eqref{eqn: QI},
 \begin{align*}
 \ud(X^F_t X^G_t)
 &=
 \big[
( I_{\khat} \ot X^F_t ) ( F \ot I_\Fock ) ( I_{\khat} \ot X^G_t )
  +
 ( I_{\khat} \ot X^F_t X^G_t ) ( G \ot I_\Fock )
\\
&
\qquad \qquad \qquad \qquad
+
 ( I_{\khat} \ot X^F_t ) ( F \ot I_\Fock )  ( I_{\khat} \ot X^G_t ) ( \Delta G \ot I_\Fock )
   \big]  \ud\Lambda_t
  \\
 &=
 X^F_t X^G_t \cdot (F \comp G) \ud\Lambda_t.
 \end{align*}
 Since $X^F_0 X^G_0 = I$
and
$F\comp G \in \Qqclg$,
 uniqueness for
strongly continuous bounded solutions of
the QS differential equation
 \[
 \ud X_t = X_t \cdot (F \comp G) \ud\Lambda_t, \quad X_0 = I_{\init \ot \Fock}
 \]
(see Theorem~\ref{thm: cocycle = qsde})
 implies that $X^F X^G = X^{F \comp G}$.
 Since $F^* \in \Qqclg$ and $G^* \in \QqclgC \ot I_\init$,
 it follows that $X^{F^* \comp G^*} = X^{F^*} X^{G^*}$ too.
 The second identity now follows by duality:
 \[
X^{ G \comp F}
=
X^{( F^* \comp G^* )^*}
=
\big( X^{ F^* \comp G^* } \big)^{\dual}
=
\big( X^{ F^*} X^{G^* } \big)^{\dual}
=
\big( X^{ G^* } \big)^{\dual} \ \big( X^{ F^* } \big)^{\dual}
=
X^G\, X^F.
 \]
\end{proof}

\begin{rems}
(i)
More generally,
if two quasicontractive QS cocycles \emph{commute on their initial space}
then
their (pointwise) product is also a cocycle.

(ii)
 In the light of the identities
 \[
 F_c \comp F_d = F_{c+d} -i \im \ip{c}{d} \Delta^\perp
 \text{ and }
 F_{-c} = F_c^*
 \qquad
 (c,d\in\noise),
 \]
 Proposition~\ref{propn: VW}
 contains Weyl commutation relations (\emph{see}~\cite{BrR}, or~\cite{L})
as a special case.
\end{rems}

The connection to quantum random walks mentioned in the introduction is as follows
(for details, in particular the precise meaning of  the terminology used,
see~\cite{BGL}).

\begin{thm}[\cite{BGL}]
For $i = 1,2$, let $F_i \in \Qqclg$ and
let $( G_i(h) )_{h > 0}$ be a family in $B( \khat \ot \init )$ satisfying
\begin{align*}
&
\sup \big\{
\norm{ G_i(h) }^n:
n \in \Nat, h > 0,
nh \les T
\big\}
< \infty
\qquad
(T \in \Rplus),
\ \text{ and }
%\end{align*}
%and
%\begin{align*}
\\
&
\big( h^{-1/2} \Delta^\perp + \Delta \big) \big( G_i(h) - \Delta^\perp \big) \big( h^{-1/2} \Delta^\perp + \Delta \big)
\to F_i + \Delta
\ \text{ as } h \to 0.
\end{align*}
Then, for all $\varphi \in B(\Fock)_*$ and $T \in \Rplus$,
\[
 \sup_{0 \les t \les T}
 \big\|
 \big( \id_{B(\init)} \otol \varphi \big)
 \big( \Xh_t - X^{ F_1 \comp F_2}_t \big)
 \big\|
 \to 0
 \ \text{ as } \
 h \to 0,
 \]
 where,
 for $h>0$,
  $( \Xh_t )_{t \ges 0}$ denotes the
 $h$-scale embedded left quantum random walk generated by $G_1(h)G_2(h)$.
\end{thm}

  \section{Product Formula}

 \label{section: product formula}

 For proving the quantum stochastic Lie--Trotter product formula it
 is convenient to define a constant associated with a pair of quantum
 stochastic generators.
 Thus, for
 $F_1, F_2 \in \Qblg$ with respective block matrix forms
 $\left[ \begin{smallmatrix}
  K_1 & M_1 \\ L_1  & N_1
 \end{smallmatrix} \right]$ and
  $\left[ \begin{smallmatrix}
  K_2 & M_2 \\ L_2  & N_2
 \end{smallmatrix} \right]$,
 set
 \begin{align*}
 \nonumber
 C(F_1, F_2)
 :&=
 \norm{\Delta^\perp F_1 \Delta^\perp} +
 \norm{\Delta^\perp F_2 \Delta^\perp} +
 \norm{\Delta^\perp F_1 \Delta} \cdot
 \norm{\Delta F_2 \Delta^\perp}
 \\
 &=
% \label{eqn: C}
 \norm{K_1} + \norm{K_2} + \norm{M_1} \norm{L_2}.
 \end{align*}
 Note that this depends only on the first row of $F_1$ and first
 column of $F_2$:
 \[
  C(F_1, F_2) =  C(\Delta^\perp F_1, F_2 \Delta^\perp).
 \]

 \begin{propn}
% \label{Proposition: One}
 \label{propn: EV}
  Let
  $F_i
  \in \Qqclgbetai$
 with block matrix form
 $\left[ \begin{smallmatrix}
  K_i & M_i \\ L_i & N_i
  \end{smallmatrix} \right]$,
   for $i=1,2$.
  Then
  \[
 \big\|
  \Expectation \big[
  X^{F_1}_t X^{F_2}_t  - X^{F_1 \comp F_2}_t
  \big]
 \big\|
 \les t^2 \, e^{t(\beta_1 + \beta_2)} C(F_1, F_2)^2
 \qquad
 (t \in \Rplus).
  \]
  \end{propn}

 \begin{proof}
 Set $\Vone = X^{F_1}$, $\Vtwo = X^{F_2}$ and $V = X^{F_1 \comp F_2}$.
 Since
 \[
 e^{-t(\beta_1 + \beta_2)} (\Vone_t \Vtwo_t - V_t) =
 X^{G_1}_t X^{G_2}_t - X^{G_1 \comp G_2}_t
 \qquad
 (t \in \Rplus),
 \]
 where $G_i = F_i - \beta_i \Delta^\perp \in \Qclg$
 ($i = 1,2$),
 it suffices to assume that $\beta_1 = \beta_2 = 0$,
 so that
$F_1, F_2, F_1 \comp F_2  \in \Qclg$, and
the QS cocycles $\Vone$, $\Vtwo$ and $V$
are all contractive.

 Fix $t \in \Rplus$.
 By the quantum It\^o product formula~\eqref{eqn: QI},
 $\Vone \Vtwo - V = \Lambda_\cdot(H)$ where
 $H = \Hone + \Htwo + \Hthree$
for the processes given by
 \begin{align*}
 &\Hone_s :=
 (I_{\khat} \ot \Vone_s) (F_1 \ot I_\Fock)  (I_{\khat} \ot \Vtwo_s)
 -
  (I_{\khat} \ot V_s) (F_1 \ot I_\Fock)
 \\
 &\Htwo_s :=
 \big( I_{\khat} \ot ( \Vone_s \Vtwo_s - V_s ) \big)
 (F_2 \ot I_\Fock)
 \\
 &\Hthree_s :=
 (I_{\khat} \ot \Vone_s) (F_1 \Delta \ot I_\Fock)
 (I_{\khat} \ot \Vtwo_s) (\Delta F_2 \ot I_\Fock)
 -
  (I_{\khat} \ot V_s) (F_1 \Delta F_2 \ot I_\Fock).
 \end{align*}
Therefore,
by~\eqref{eqn: expect},
 $\Expect [ \Vone_t \Vtwo_t - V_t ] =
 \int_0^t \ud s \, \Expect [ Y_s ]$
where
$Y = \Yone + \Ytwo + \Ythree$
for the processes given by
 \begin{align*}
 &\Yone_s :=
 \Vone_s ( K_1 \ot I_\Fock)  \Vtwo_s
 -
   V_s (K_1 \ot I_\Fock)
 \\
 &\Ytwo_s :=
  ( \Vone_s \Vtwo_s - V_s )
 (K_2 \ot I_\Fock)
 \\
 &\Ythree_s :=
 \Vone_s  (M_1 \ot I_\Fock)
 (I_{\noise} \ot \Vtwo_s) ( L_2 \ot I_\Fock)
 -
  V_s  (M_1 L_2 \ot I_\Fock).
 \end{align*}
Note that,
since $\Vone$, $\Vtwo$ and $V$ are contractive,
\begin{equation}
 \label{eqn: EYr}
\norm{\Expect [ Y_r ] } \les 2 \big( \norm{K_1} + \norm{K_2} + \norm{M_1} \norm{L_2} \big)
=
2 C(F_1, F_2)
\qquad
(r \in \Rplus).
\end{equation}
In turn,
the quantum It\^o formula, together with
the identities~\eqref{eqn: expect} and~\eqref{eqn: under},
imply that, for $i = 1, 2, 3$ and $s \in \Rplus$,
$\Expect [ \Yi_s ]  =
 \int_0^s \ud r \, \Expect [ \Zi_r ]$
for processes given by
 \begin{align*}
 \Zone_r :=
\
&
% \big(
 \Vone_r ( (K_1)^2 \ot I_\Fock ) \Vtwo_r - V_r ((K_1)^2 \ot I_\Fock)
% \big)
+
 \\
&
 \Vone_r (K_1 \ot I_\Fock) \Vtwo_r ( K_2 \ot I_\Fock )
 -
 V_r (K_1 K_2 \ot I_\Fock)
+
 \\
 &
 \Vone_r (M_1 \ot I_\Fock)
 (I_{\noise} \ot K_1 \ot I_\Fock) (I_{\noise} \ot  \Vtwo_r) (L_2 \ot I_\Fock)
%\\
% &\qquad
% \qquad \qquad \qquad \qquad \qquad \qquad \qquad \qquad \qquad \qquad
 -
 V_r (M_1 L_2 K_1 \ot I_\Fock),
 \\
 \Ztwo_r :=
\
&
 Y_r (K_2 \ot I_\Fock), \text{ and }
 \\
\Zthree_r :=
\
&
 \Vone_r
 (K_1 M_1 \ot I_\Fock) (I_\noise \ot \Vtwo_r ) (L_2 \ot I_\Fock) \Vtwo_r
 - V_r ( K_2 M_1 L_2 \ot I_\Fock )
+
\\
 &
 \Vone_r
 ( M_1 \ot I_\Fock) (I_\noise \ot \Vtwo_r )
 (I_{\noise} \ot K_2 \ot I_\Fock) (L_2 \ot I_\Fock)
 % \\
 % &\qquad
 % \qquad \qquad \qquad \qquad \qquad \qquad \qquad \qquad \qquad \qquad
  -
  V_r ( K_2 M_1 L_2 \ot I_\Fock ) +
 \\
 &
 \Vone_r (M_1 \ot I_\Fock)
 ( I_\noise \ot M_1 \ot I_\Fock)
 (I_{\noise \ot \noise} \ot \Vtwo_r)
 (I_\noise \ot L_2 \ot I_\Fock)
 (I_{\noise \ot \noise} \ot \Vtwo_r)
  ( L_2 \ot I_\Fock )
 \\
 & \quad
\qquad \qquad
  \qquad \qquad \qquad \qquad \qquad \qquad \qquad \qquad \qquad \qquad
 -
 V_r ((M_1 L_2)^2 \ot I_\Fock).
 \end{align*}
Now,
 by further use of the contractivity of the processes
$\Vone$, $\Vtwo$ and $V$, together with the estimate~\eqref{eqn: EYr},
 \begin{align*}
\norm{ \Expect [ \Zone_r ] }
\les
& 2
 \big(
 \norm{K_1}^2  +
 \norm{ K_1 } \norm{  K_2 } +
\norm{ M_1 } \norm{ K_1 } \norm{ L_2 }
\big)
=
2 \norm{ K_1} C(F_1, F_2),
\\
\norm{ \Expect [ \Ztwo_r ] } \les
& 2   C(F_1, F_2)
 \norm{K_2},
\ \text{ and }
\\
\norm{ \Expect [ \Zthree_r ] } \les
& 2
 \big(
\norm{K_1} \norm{ M_1 } \norm{ L_2 }
  +
\norm{ M_1 } \norm{ K_2 } \norm{ L_2 }
+
\norm{ M_1 }^2 \norm{ L_2 }^2
\big)
=
2 \norm{ M_1 } \norm{ L_2 } C(F_1, F_2),
  \end{align*}
and so
\begin{align*}
\norm{ \Expect [ \Vone_t \Vtwo_t - V_t ] }
&
=
\Big\| \int_0^t \! \ud s \, \int_0^s \! \ud r \, \Expect[\Zone_r + \Ztwo_r + \Zthree_r]\  \Big\|
\\
&
\les
\int_0^t \! \ud s \, \int_0^s \! \ud r \, 2 \, C(F_1, F_2)^2
=
t^2 \, C(F_1, F_2)^2,
\end{align*}
 as required.
\end{proof}

\begin{lemma}
% \label{Proposition: One}
 \label{lemma: wci}
  Let
  $F_1  \in \Qqclgbetaone$,
    $F_2  \in \Qqclgbetatwo$,
 $c', c \in \noise$
and
 $r, t \in \Rplus$ with $t > r$.
Then
  \begin{equation*}
 \big\|
 \Omega\big( c'_{[r,t[} ,  c_{[r,t[} \big)
 \big(
  X^{F_1}_{r,t} X^{F_2}_{r,t} - X^{F_1 \comp F_2}_{r,t}
  \big)
 \big\|
 \les ( t - r )^2 \, e^{( t - r )(\beta_1 + \beta_2)} \,
 C\big( F_{c'}^* \comp F_1, F_2 \comp F_c \big)^2.
  \end{equation*}
  \end{lemma}

 \begin{proof}
By definition of the evolutions
$\big( X^G_{r,t} \big)_{0\les r \les t}$, for $G= F_1, F_2$ and $F_1 \comp F_2$,
and identity~\eqref{subeqn: QSEc} along with the remark following it,
\[
\LHS
=
\big\|
 \Omega\big( c'_{[0, t-r[} ,  c_{[0, t-r[} \big)
 \big(
  X^{F_1}_{t-r} X^{F_2}_{t-r} - X^{F_1 \comp F_2}_{t-r}
  \big)
 \big\|
\]
and so we may suppose without loss of generality that $r = 0$.
 In this case, by the identity~\eqref{eqn: slices},
Example~\ref{examp: Weyl}
 and the associativity of $\comp$,
 the estimate follows immediately from Propositions~\ref{propn: VW}
 and~\ref{propn: EV}.
 \end{proof}

Now let $\Step$ denote the subspace of
$L^2(\Rplus; \noise)$ consisting of step functions,
whose right-continuous versions we use when evaluating.
For $F_1, F_2 \in \Qblg$
and $g', g \in \Step$,
set
\[
C( g', F_1, F_2, g ) :=
\max \Big\{
 C\big( F_{c'}^* \comp F_1,\, F_2 \comp F_c \big):
 c' \in \Ran g', c \in \Ran g \Big\}.
\]

\begin{lemma}
% \label{Proposition: One}
 \label{lemma: wciplus}
  Let
  $F_1  \in \Qqclgbetaone$,
    $F_2  \in \Qqclgbetatwo$,
 $g', g \in \Step$
and
 $r, t \in \Rplus$ with $t > r$.
Then
  \begin{equation*}
 \big\|
 \Omega\big( g'_{[r,t[} ,  g_{[r,t[} \big)
 \big(
  X^{F_1}_{r,t} X^{F_2}_{r,t} - X^{F_1 \comp F_2}_{r,t}
  \big)
 \big\|
 \les ( t - r )^2 \, e^{( t - r )(\beta_1 + \beta_2)} \,
  C( g'_{[r,t[}, F_1, F_2, g_{[r,t[})^2.
  \end{equation*}
  \end{lemma}
\begin{proof}
 Set $\beta := \beta_1 + \beta_2$
 and let
 $\{ r = t_0 < t_1 < \cdots < t_n < t_{n+1} = t \}$
 be such that $g$ and $g'$ are constant,
 say $c^i$ and $d^i$ respectively, on the interval $[ t_i, t_{i+1} [$,
 and,
  for $i = 0, \cdots , n$,
 set
 $\Omega_i := \Omega( g'_{[ t_i, t_{i+1} [}, g_{[ t_i, t_{i+1} [} )$
 \[
 X^1_i := X^{F_1}_{t_i, t_{i+1}},
 X^2_i := X^{F_2}_{t_i, t_{i+1}} \text{ and }
 X_i := X^{F}_{t_i, t_{i+1}},
 \]
 noting that
 $\Omega_i := \Omega( c^i_{[ t_i, t_{i+1} [}, d^i_{[ t_i, t_{i+1} [} )$.
 Then,
 again using identity~\eqref{subeqn: QSEd}
 and the subsequent remark,
 we see that
 $$
 \Omega\big( g'_{[r,t[} ,  g_{[r,t[} \big)
 \big(
  X^{F_1}_{r,t} X^{F_2}_{r,t} - X^{F_1 \comp F_2}_{r,t}
  \big)
=
 \prod_{0 \les i \les n}^{\longrightarrow} \Omega_i( X^1_i X^2_i )
 -
  \prod_{0 \les i \les n}^{\longrightarrow} \Omega_i( X_i ).
  $$
  Therefore, using Lemma~\ref{lemma: wci},
\begin{align*}
\LHS
=
&
\Big\|
\sum_{i=0}^n
\Big(
 \prod_{0 \les p < i}^{\longrightarrow} \Omega_p(X_p)
\ \
\Omega_i ( X^1_i X^2_i - X_i )
 \prod_{i < p \les n}^{\longrightarrow}
\Omega_p( X^1_p X^2_p )
\Big)
\Big\|
\\
\les
 &
\sum_{i=0}^n
e^{ \beta ( t_i - r ) }
\big\|
\Omega_i ( X^1_i X^2_i - X_i )
\big\|
e^{ \beta ( t - t_{i+1} ) }
\\
\les
&
e^{ \beta ( t - r ) }
\sum_{i=1}^n
 (t_{i+1} - t_i)^2
 C\big( F_{c^i}^* \comp F_1, F_2 \comp F_{d^i} \big)^2
 \les
\RHS.
\end{align*}
\end{proof}

By a \emph{partition} $\mathcal{P}$ of $\Rplus$ we mean
a sequence $(s_n)$ in $\Rplus$
which is strictly increasing and tends to infinity.
Where convenient we identify a partition with its set of terms.
For $S \subset \subset \Rplus$,
let $| S |$ denote the  mesh of $\{ 0 \} \cup S$,
that is, $\max \{ | t_{i} - t_{i-1}|: \, 1 \les i \les k \}$,
where
$S = \{ t_1 < \cdots < t_n \}$ and $t_0 := 0$.

 \begin{lemma}
 \label{lemma: wgfN}
  Let
  $F_1  \in \Qqclgbetaone$,
    $F_2  \in \Qqclgbetatwo$,
 $g', g \in \Step$
and
 $r, t \in \Rplus$ with $t > r$.
 Let
 $\mathcal{P}$ be a partition of $\Rplus$,
  let
 $\{ t_1 < \cdots < t_{N-1} \} = \mathcal{P}\, \cap \ ]r,t[$
and set $t_0 := r$ and $t_N := t$.
Then
  \begin{multline*}
 \Big\|
 \Omega\big( g'_{[r,t[}, g_{[r,t[} \big)
 \Big(
\prod_{1 \les j \les N}^{\longrightarrow}
  X^{F_1}_{t_{j-1},t_j} X^{F_2}_{t_{j-1},t_j}  -
  X^{F_1 \comp F_2}_{r,t}
  \Big)
 \Big\|
 \\
 \les  \big| \mathcal{P} \cap [r,t] \big|
 \, (t-r)\, e^{(t-r)(\beta_1 + \beta_2)} \,
 C( g'_{[r,t[}, F_1, F_2, g_{[r,t[})^2.
  \end{multline*}
    \end{lemma}

 \begin{proof}
 Set $\beta := \beta_1 + \beta_2$ and,
for $j = 1, \cdots , N$, set
$\Omega_j :=  \Omega\big( g'_{[t_{j-1}, t_j[}, g_{[t_{j-1}, t_j[} \big)$,
\[
X^1_j := X_{t_{j-1}, t_j}^{F_1}, \
X^2_j := X_{t_{j-1}, t_j}^{F_1} \ \text{ and } \
X_j := X^{F_1 \comp F_2}_{t_{j-1}, t_j}.
\]
Then, using  Lemma~\ref{lemma: wciplus}
(and arguing as in its proof), we see that
\begin{align*}
\LHS
=
&
\Big\|
\sum_{j=1}^N
\Big(
 \prod_{1 \les n < j}^{\longrightarrow} \Omega_n(X_n)
\ \
\Omega_j ( X^1_j X^2_j - X_j )
 \prod_{j < n \les N}^{\longrightarrow}
\Omega_n( X^1_n X^2_n )
\Big)
\Big\|
\\
\les
 &
\sum_{j=1}^N
e^{ \beta ( t_{j-1} - r ) }
\big\|
\Omega_j ( X^1_j X^2_j - X_j )
\big\|
e^{ \beta ( t - t_{j+1} ) }
\\
\les
&
e^{ \beta ( t - r ) }
\sum_{j=1}^N
 (t_j - t_{j-1})^2
 C( g'_{[t_{j-1}, t_j[}, F_1, F_2, g_{[t_{j-1}, t_j[})^2
\\
\les
&
e^{ \beta ( t - r ) }
  \big| \mathcal{P} \cap [r,t] \big| \, (t-r)\, C( g'_{[r,t[}, F_1, F_2, g_{[r,t[})^2,
\end{align*}
as required.
\end{proof}

 We may now prove the main result.

 \begin{thm}
 \label{thm: the theorem}
 Let $V^1$ and $V^2$  be quasicontractive elementary QS
 cocycles on $\init$ with noise dimension space
 $\noise$
 and respective stochastic generators $F_1$ and $F_2$,
and let
$T \in \Rplus$ and $\varphi \in B(\Fock)_*$.
Then
\[
 \sup_{0 \les r \les t \les T}
 \big\|
 \big( \id_{B(\init)} \otol \varphi \big)
 \big( V^{ \mathcal{P}, 1 \comp 2 }_{r,t} - X^{ F_1 \comp F_2}_{r,t} \big)
 \big\|
 \to 0
 \ \text{ as } \
 \big| \mathcal{P} \cap [0,T] \big| \to 0.
 \]
The notation here is as follows.
For
 a partition $\mathcal{P}$ of $\Rplus$ and $0\les r \les t$,
 \[
V^{ \mathcal{P}, 1 \comp 2 }_{r,t} :=
 \prod_{1 \les j \les N}^{\longrightarrow}
 V^1_{t_{j-1},t_j} V^2_{t_{j-1},t_j}
 \]
in which
$ \{ t_1 < \cdots < t_{N-1} \} =  \mathcal{P}\, \cap \ ]r,t[$, $t_0 := r$ and $t_N := t$.
 \end{thm}

 \begin{proof}
 Set
$\beta := \beta_0(V^1) + \beta_0(V^2) $
and
$V := X^{ F_1 \comp F_2}$.
We may suppose without loss of generality that
$\varphi$ is of the form $\omega_{\vp(h'), \vp(h)}$
for some $h', h \in \Step$, since such functionals are total in
the Banach space $ B(\Fock)_*$
and
$
( V^{ \mathcal{P}, 1 \comp 2 }_{r,t} - V_{r,t} )
$
is uniformly bounded by $2 e^{T \max\{ 0, \beta \} }$
for $[r, t] \subset [0, T]$.

 For any partition $\mathcal{P}$ of $\Rplus$,
 Lemma~\ref{lemma: wgfN} implies that
there is a constant $C = C(\varphi, F_1, F_2)$
such that,
 for all
 subintervals $[r,t]$ of $[0,T]$,
  \[
 \big\|
 \big( \id_{B(\init)} \otol \varphi \big)
 \big( V^{ \mathcal{P}, 1 \comp 2 }_{r,t} -
 V_{r,t} \big)
 \big\|
 \les
 \big| \mathcal{P} \cap [0,T] \big| \, T \,
 e^{T \beta }\, C^2.
\]
 The result follows.
\end{proof}

Proposition~\ref{propn: EV} and
Lemmas~\ref{lemma: wci},~\ref{lemma: wciplus} and~\ref{lemma: wgfN}
extend to the case of QS generators
$F_i = \left[ \begin{smallmatrix} K_i & M_i \\ L_i & N_i  \end{smallmatrix} \right]$
($i = 1, \cdots , n$),
the constant $C(F_1, F_2)$ being replaced by
\[
\Delta^\perp
f_1
\comp
\big( f_2 \comp \cdots \comp f_{n-1} \big) \comp
f_n \,
\Delta^\perp
\]
where
%$f_i, \Delta \in B( \wh{\Comp} )$
%are given by
\[
f_i = \begin{bmatrix} \| K_i \| & \| M_i \| \\ \| L_i \| & \| N_i \| \end{bmatrix}
\in B( \wh{\Comp} )
\ \text{ and } \
\Delta^\perp =
\Delta^\perp_{\Comp} = \begin{bmatrix} 1 & 0 \\ 0 & 0 \end{bmatrix}
\in B( \wh{\Comp} ).
%\in B( \wh{\Comp} ).
\]
In this case the quantum It\^o product formula allows
$X^{F_1}_t \cdots X^{F_n}_t - I$
to be expressed as a sum of $2^{n} - 1$ QS integrals,
each of which may be paired with one of the $2^n - 1$ terms arising from
the expansion of $F_1 \comp \cdots \comp F_n$ in
the identity
\[
X^{F_1 \comp \cdots \comp F_n}_t - I
=
\int_0^t
X^{F_1 \comp \cdots \comp F_n}_s
\cdot
\big( F_1 \comp \cdots \comp F_n \big) \ud \Lambda_s.
\]
As a consequence,
Theorem~\ref{thm: the theorem} extends to
quasicontractive elementary QS cocycles $V^1, \cdots , V^n$,
with respective stochastic generators $F_1, \cdots , F_n$,
as follows:
\[
 \sup_{0 \les r \les t \les T}
 \big\|
 \big( \id_{B(\init)} \otol \varphi \big)
 \big( V^{ \mathcal{P}, 1 \comp \cdots \comp n }_{r,t} - X^{ F_1 \comp \cdots \comp F_n}_{r,t} \big)
 \big\|
 \to 0
 \ \text{ as } \
 \big| \mathcal{P} \cap [0,T] \big| \to 0,
 \]
 in the notation
 \[
V^{ \mathcal{P}, 1 \comp \cdots \comp n }_{r,t} :=
 \prod_{1 \les j \les N}^{\longrightarrow}
V^1_{t_{j-1},t_j} \cdots V^n_{t_{j-1},t_j}.
 \]

 \begin{rems}
(i)
 Note that the QS cocycle $V =  X^{ F_1 \comp \cdots \comp F_n}$ is
 contractive (respectively, isometric, or coisometric)
 provided that all of $V^1, \cdots , V^n$ are.

(ii)
  Whilst the convergence of Trotter products holds in the above
  hybrid norm-ultraweak topology,
 if
  $V$ is isometric and $V^1, \cdots , V^n$ are contractive
  (in particular, if $V^1, \cdots , V^n$ are isometric)
  then convergence also holds in the strong operator topology:
  \[
 \sup_{0 \les r \les t \les T}
 \big\|
 \big(
V^{ \mathcal{P}, 1 \comp \cdots \comp n }_{r,t} -
 X_{r,t}^{F_1 \comp \cdots \comp F_n}
\big)\xi \big\|  \to 0
 \ \text{ as } \
 \big|\mathcal{P} \cap [0,T] \big| \to 0
 \quad
 ( T \in \Rplus, \xi \in \init \ot \Fock ).
  \]
  This follows from the uniform continuity of the function
  \[
  [0, T]^2_{\les} \to \init \ot \Fock, \quad
  (r,t) \mapsto V_{r,t} \xi
  \qquad
  (\xi \in \init \ot \Fock, T \in \Rplus).
  \]
 \end{rems}

 We now revisit the case of QS cocycles with independent driving noise
 and show how to view it as a special case of the above theorem.
 Suppose therefore that the noise dimension space has orthogonal
 decomposition $\noise = \noise_1 \op \cdots \op \noise_n$.
 Denoting the symmetric Fock
 space over $L^2(\Rplus; \noise_i)$ by $\Fock^i$,
 second quantisation of the natural isometry
 from $L^2(\Rplus; \noise_i)$ to $L^2(\Rplus; \noise)$
 (see~\eqref{eqn: second quant}),
followed by ampliation, gives an isometry
 $\init \ot \Fock^i \to \init \ot \Fock$
 which in turn induces a normal *-algebra monomorphism
 \[
 \Upsilon^i: B(\init \ot \Fock^i) \to B(\init \ot \Fock)
 %\ \text{ and } \
 %\Upsilontwo: B(\init \ot \Fock^2) \to B(\init \ot \Fock).
\qquad
 (i=1, \cdots , n).
 \]

 \begin{cor}[\cite{LS L-T}]
 \label{cor: LS L-T}
 For $i=1, \cdots , n$,
 let $V^i$ be a quasicontractive elementary QS
 cocycle on $\init$ with noise dimension space
$\noise_i$
and stochastic generator $F_i$.
 Then,
 for all $T \in \Rplus$ and $\varphi \in B(\Fock)_*$,
\[
 \sup_{0 \les r \les t \les T}
 \big\|
 \big( \id_{B(\init)} \otol \varphi \big)
 \big( V^{ \mathcal{P}, 1 \boxplus \, \cdots \, \boxplus n }_{r,t} -
 X^{F_1 \boxplus \, \cdots \, \boxplus F_n}_{r,t} \big)
 \big\|
 \to 0
 \ \text{ as } \
 \big| \mathcal{P} \cap [0,T] \big| \to 0.
 \]
 The notation here is as follows.
For
 a partition $\mathcal{P}$ of $\Rplus$ and $0\les r \les t$,
 setting
  $\{ t_1 < \cdots < t_{N-1} \} = \mathcal{P}\, \cap \ ]r,t[$,
$t_0 := r$ and $t_N := t$,
 \[
V^{ \mathcal{P}, 1 \boxplus \, \cdots \, \boxplus n }_{r,t} :=
 \prod_{1 \les j \les N}^{\longrightarrow}
 \Upsilonone ( V^1_{t_{j-1},t_j} ) \cdots \Upsilonen ( V^n_{t_{j-1},t_j} ),
 \]
% in which
% $\{ t_1 < \cdots < t_{N-1} \} = \mathcal{P}\, \cap \ ]r,t[$,
%$t_0 := r$ and $t_N := t$.
 \end{cor}

 \begin{proof}
 Suppose first that $n=2$.
 Note that,
in the notation~\eqref{eqn: iota con},
 $\Upsilonone( V^1_{s,s'} ) = X_{s,s'}^{\iota(F_1)}$ and
 $\Upsilontwo( V^2_{s,s'} ) = X_{s,s'}^{\iota'(F_2)}$ ($s,s'\in\Rplus, s<s'$).
Therefore,
in view of identity~\eqref{eqn: effectively},
 for this case
 the corollary follows immediately from
 Theorem~\ref{thm: the theorem}.
 The general case follows similarly via
 obvious extension of the notation~\eqref{eqn: iota con} and corresponding identity~\eqref{eqn: effectively}.
 \end{proof}

%\begin{rem}
%This also naturally extends to the case of $n$ quasicontractive Markov-regular QS cocycles.
%\end{rem}

%%%%%%%%%%%%%%%%%%%%%%%%%%%%%%%%%%%%%%%%%%%%%%%%%%%
% NEW SECTION
%%%%%%%%%%%%%%%%%%%%%%%%%%%%%%%%%%%%%%%%%%%%%%%%%%%

 \section{Maximal Gaussian component of a QS generator}
 \label{section: gaussian}

In this section it is shown that every QS generator $F \in \Qqclg$ enjoys a unique decomposition
$F_1 \boxplus F_2$ in which $F_2$ is `pure Gaussian' and $F_1$ is `wholly non-Gaussian', in senses to be defined below.
This amounts to extracting a maximal Gaussian component of the generator and
demonstrating its uniqueness.
It chimes well with the way that Hunt's formula
decomposes
the generator of a L\'evy process on a compact Lie group
into
a maximal Gaussian component and a jump part (albeit not exactly uniquely).

To each contraction $C \in B(\noise \ot \init)$ we associate the following closed subspaces of $\noise$:
\[
\noise^C_{\Gauss} :=
\big\{
c \in \noise: (C - I) E_c = 0
\big\},
\text{ where }
E_c :=  \ket{c} \ot I_\init,
\]
and
\[
\noise^C_{\Pres} :=
\Linbar \,
\big\{
( I_\noise \ot \bra{u} ) \xi:
\xi \in \Ran (C - I)^*, u \in \init
\big\}.
\]
The former captures the subspace of all `directions' in which $C$ acts as the identity operator;
the latter is complementary (see below).
The subscripts denote `Gaussian' and `preservation' parts.

Letting $J^C_{\Pres}$ denote the inclusion $\noise^C_{\Pres} \to \noise$, set
\[
C_{\Pres} :=
( J^C_{\Pres} \ot I_\init )^* C ( J^C_{\Pres} \ot I_\init ),
\]
the compression of $C$ to $\noise^C_{\Pres} \ot \init$.

\begin{lemma}
\label{lemma: kCNG}
Let $C \in B(\noise \ot \init )$ be a contraction.
Then the following hold\tu{:}
\begin{alist}
\item
$\noise =  \noise^C_{\Pres} \op \noise^C_{\Gauss}$.
\item
$C$ has block matrix form
$\left[ \begin{smallmatrix} C_{\Pres} & 0 \\ 0 & I^C_{\Gauss}\end{smallmatrix} \right]$,
where
$I^C_{\Gauss} := I_{ \noise^C_{\Gauss}  \ot \init}$.
\item
$( \noise^C_{\Pres} )_{\Pres}^{C_{\Pres}} = \noise^C_{\Pres}$.
\end{alist}
\end{lemma}

\begin{proof}
Set $I^C_{\Pres} := I_{\noise^C_{\Pres} \ot \init}$.

(a)
Let $c \in \noise$.
Then, for all $u \in \init$,
$$
\ip{c}{ ( I_\noise \ot \bra{u} ) (C - I )^* \xi} = 0 \text{ for all } \xi \in \noise \ot \init
\iff
( C - I ) ( c \ot u ) = 0.
$$
It follows that $c \in ( \noise^C_{\Pres} )^\perp$
if and only if $c \in \noise^C_{\Gauss}$,
so (a) follows.

(b)
Contraction operators $T$ on a Hilbert space $\Hil$ enjoy the following elementary property.
In terms of the orthogonal decomposition
$\Hil = \Kil^\perp \op \Kil$ where $\Kil := \Ker (T - I )$,
$T$ has diagonal block matrix form
$\left[ \begin{smallmatrix} * & 0 \\ 0 & I_\Kil \end{smallmatrix} \right]$.
Since
$\noise^C_{\Gauss}  \ot \init \subset \Ker ( C - I )$
it follows that, in terms of the orthogonal decomposition
$
\noise \ot \init =
 ( \noise^C_{\Pres} \ot \init ) \op ( \noise^C_{\Gauss} \ot \init )
$,
$C$ has the claimed block matrix form.

(c)
By (a) it suffices to show that $( \noise^C_{\Pres} )^{C_{\Pres}}_{\Gauss} = \{ 0 \}$.
Accordingly let $e \in ( \noise^C_{\Pres} )^{C_{\Pres}}_{\Gauss}$.
Then $( C_{\Pres} - I^C_{\Pres} ) E_e = 0$
and so, by (b),
$(C - I ) E_{J^C_{\Pres} e} = 0$,
in other words
$J^C_{\Pres} e \in \noise^C_{\Gauss}$.
Thus $J^C_{\Pres} e \in \noise^C_{\Pres} \cap \noise^C_{\Gauss} = \{ 0 \}$,
by (a).
Since $J^C_{\Pres}$ is an inclusion this implies that $e = 0$, as required.
\end{proof}

\begin{defn}
Let $F \in \Qqclg$ with block matrix form $\left[ \begin{smallmatrix} * & * \\ * & C - I\end{smallmatrix}\right]$.
We say that
\begin{alist}
\item
$F$ is
\emph{Gaussian} if
$\noise^C_{\Gauss} = \noise$,
equivalently $C = I$ so that $F$ has block matrix form
 $\left[ \begin{smallmatrix} * & * \\ * & 0 \end{smallmatrix}\right]$;
\item
$F$ is
\emph{wholly non-Gaussian} if
$\noise^C_{\Pres} = \noise$,
equivalently $C _{\Pres}= C$;
\item
$F$ is
\emph{pure Gaussian} if
it has block matrix form
 $\left[ \begin{smallmatrix} -\tfrac{1}{2} L^*L & -L^* \\ L & 0 \end{smallmatrix}\right]$,
for some $L \in B(\init ; \noise \ot \init)$,
equivalently
$\left[ \begin{smallmatrix} -\tfrac{1}{2} M M^* & M \\ -M^*  & 0 \end{smallmatrix}\right]$,
for some $M \in B( \noise \ot \init; \init )$;
\item
$F$ is
\emph{pure preservation} if
it has block matrix form
$\left[ \begin{smallmatrix} 0 & 0 \\ 0 & C - I \end{smallmatrix}\right]$;
\item
$F$ is
\emph{pure drift} if
it has block matrix form
$\left[ \begin{smallmatrix} * & 0 \\ 0 & 0 \end{smallmatrix}\right]$,
in other words $F \in \Qzlg$.
\end{alist}
Write
\[
\Qglg, \ \Qwnglg, \ \Qpglg  \, \text{ and } \, \Qpplg,
\]
for the respective classes of Gaussian, wholly non-Gaussian, pure Gaussian and pure preservation generator.
\end{defn}

\begin{rems}
(i)
The terminology has its origins in the work of Sch\"urmann on
L\'evy--Khintchin-type decompositions
for L\'evy processes on bialgebras
such as algebraic quantum groups (\cite{Schurmann}).
In the presence of a minimality condition,
the Gaussian property for the generator of an elementary unitary QS cocycle $U$
may alternatively be expressed in terms of the cocycle itself as follows:
\[
t^{-1}
\Expect \big[
( U^1_t - I )^{u_1}_{v_1} ( U^2_t - I )^{u_2}_{v_2} ( U^3_t - I )^{u_3}_{v_3}
\big]
\to 0
\ \text{ as } \ t \to 0
\]
for all
$u_1, \cdots , v_3 \in \init$ and
all choices of $U^1$, $U^2$ and $U^3$ from $\{ U, U^* \}$
where,
for $X \in B(\init \ot \Fock)$ and $u,v \in \init$,
$
X^u_v :=
( \omega_{u,v} \otol \id_{B(\Fock)} )(X) =
( \bra{u} \ot I_\Fock ) X ( \ket{v} \ot I_\Fock )
$
(see~\cite{SSS} and~\cite{Schurmann}).

(ii)
By Theorem~\ref{thm: LW, GLSW} (a),
$F \in \Qglg$ if and only if it has block matrix form
 $\left[ \begin{smallmatrix} K & -L^* \\ L & 0 \end{smallmatrix}\right]$,
for some $K \in B(\init)$ and $L \in B(\init ; \noise \ot \init)$.

(ii)
Given a Gaussian QS generator
 $\left[ \begin{smallmatrix} K & -L^* \\ L & 0 \end{smallmatrix}\right]$,
any orthogonal decomposition $\noise = \noise_1 \op \noise_2$
determines a decomposition
$F = F_1 \boxplus F_1'$
in which,
corresponding to the block matrix decomposition
 $L = \left[ \begin{smallmatrix} L_1\\ L_2 \end{smallmatrix}\right]$,
$F_1$ is pure Gaussian with
\[
F_1 =
\begin{bmatrix}
 -\tfrac{1}{2} L_1^*L_1 & -L_1^* \\ L_1 & 0
\end{bmatrix}
\ \text{ and } \
F_1' =
\begin{bmatrix}
 K + \tfrac{1}{2} L_1^* L_1
&
-L_2^*
\\
L_2
&
0
\end{bmatrix}.
\]
Since
$L^* L = L_1^* L_1 + L_2^* L_2$,
$F_1'$ is pure Gaussian too
if and only $F$ is.

(iii)
In the case of no noise ($\noise = \{ 0 \}$), $\Qpplg = \emptyset$ and
$\Qglg$, $\Qwnglg$, $\Qpglg$ and $\Qzlg$ all equal $\Qblg$.
Otherwise, when $\noise \neq \{ 0 \}$,
all of
$\Qglg$, $\Qwnglg$, $\Qpglg$, $\Qpplg$  and $\Qzlg$
are selfadjoint and the following relations are easily verified:
\begin{align*}
&
\Qglg \cap \Qwnglg = \emptyset;
\\
&
\Qpglg \cap \Qzlg = \{ 0 \} = \Qglg \cap \Qpplg;
\\
&
\Qpglg  \subset \Qulg
\ \text{ and } \
\Qpplg \subset \Qclg;
\\
&
\Qglg \comp \Qglg = \Qglg
\ \text{ and } \
 \Qpplg \comp \Qpplg = \Qpplg;
%\ \text{ and},
\\
&
\Qzlg \comp \Qpglg = \Qglg = \Qpglg \comp \Qzlg.
\end{align*}
Also,
in view of the series decompositions (Proposition~\ref{propn: canon series}),
\begin{align*}
&
\Qzlg \comp \Qpglg \comp \Qpplg \supset \Qilg,
\ \text{ and}
\\
&
\Qpplg \comp \Qpglg \comp \Qzlg \supset \Qilg^*.
\end{align*}
The next result implies that
\[
\Qqclg =
\bigcup_{\noise = \noise_1 \op \noise_2}
\QwnglgSolo(\wh{\noise_1} \ot \init )
\boxplus
\QpglgSolo(\wh{\noise_2} \ot \init ).
\]
\end{rems}

\begin{thm}
\label{thm: G decomp}
Let $F \in \Qqclg$.
Then $F$ enjoys a unique decomposition
\[
F_{\wnGauss} \boxplus F_{\mGauss}
\]
where,
for some orthogonal decomposition $\noise_1 \op \noise_2$ of $\noise$,
$$
F_{\wnGauss} \in \QwnglgSolo( \wh{\noise_1} \ot \init )
\ \text{ and } \
F_{\mGauss} \in \QpglgSolo ( \wh{\noise_2} \ot \init ).
$$
\end{thm}

\begin{proof}
Let
$\left[ \begin{smallmatrix} K & M \\ L & C - I \end{smallmatrix}\right]$
be the block matrix decomposition of $F$,
and set
\[
L^C_{\Gauss} := ( J^C_{\Gauss} \ot I_\init )^* L, \
L^C_{\Pres} := ( J^C_{\Pres} \ot I_\init )^* L, \
M^C_{\Gauss} := M ( J^C_{\Gauss} \ot I_\init ) \ \text{ and } \
M^C_{\Pres} := M ( J^C_{\Pres} \ot I_\init ).
\]
Then
$
F = F_{\wnGauss} \boxplus F_{\mGauss}
$
where
\[
F_{\wnGauss} =
\begin{bmatrix}
K + \tfrac{1}{2} ( L^C_{\Gauss} )^* L^C_{\Gauss} & M^C_{\Pres}
\\
L^C_{\Pres} & C_{\Pres} - I
\end{bmatrix}
\
\text{ and }
\
F_{\mGauss} =
\begin{bmatrix}
- \tfrac{1}{2} ( L^C_{\Gauss} )^* L^C_{\Gauss} & M^C_{\Gauss}
\\
L^C_{\Gauss} & 0
\end{bmatrix}.
\]
By Remark (ii) above,
$M^C_{\Gauss} = - ( L^C_{\Gauss} )^*$
so
$F_{\mGauss}$ is pure Gaussian.
It follows from Lemma~\ref{lemma: kCNG}
that $F_{\wnGauss}$ is wholly non-Gaussian.
This proves existence.

For uniqueness, suppose that $F = F_1 \boxplus F_2$ where
$F_1 \in \QwnglgSolo(\wh{\noise_1} \ot \init)$ and
$F_2 \in \QpglgSolo(\wh{\noise_2} \ot \init)$
for an orthogonal decomposition $\noise = \noise_1 \oplus \noise_2$.
Then $F_1$ and $F_2$ have block matrix decompositions
\[
F_1 =
\begin{bmatrix}
K_1 & M_1
\\
L_1 & C_1 - I
\end{bmatrix}
\
\text{ and }
\
F_2 =
\begin{bmatrix}
- \tfrac{1}{2} L_2^* L_2 & -L_2^*
\\
L_2 & 0
\end{bmatrix},
\]
moreover
\[
\noise^{C_1}_{\Pres} = \noise_1
\ \text{ and } \
\noise^C_{\Gauss} \supset \noise_2.
\]
Suppose that
$e_1 \in \noise_1$ and $\binom{e_1}{0} \in \noise^C_{\Gauss}$.
Then $(C - I ) E_{\binom{e_1}{0}} = 0$,
so  $(C_1 - I ) E_{e_1} = 0$.
Since $F_1$ is wholly non-Gaussian this implies that $e_1 = 0$.
Thus $\noise^C_{\Gauss} = \noise_2$ and $\noise^C_{\Pres} = \noise_1$.
It follows that $C_1 = C_{\Pres}$ and $L_2 = L^C_{\Gauss}$.
This implies that
$F_1 = F_{\wnGauss}$ and $F_2 = F_{\mGauss}$,
as required.
\end{proof}

\begin{rems}
(i)
Clearly $F_{\mGauss}$ is the \emph{maximal} pure Gaussian component of $F$.

(ii)
Let $F \in \Qilg$.
%thus $F$ has block matrix form
%$\left[\begin{smallmatrix} iH - \tfrac{1}{2} L^* L & -L^* W \\ L & W-I \end{smallmatrix}\right]$
%where $W$ is isometric and $H$ is selfadjoint.
Then the above decomposition of $F$ takes the form
\[
F =
\begin{bmatrix} K_1 & -L_1^* W \\ L_1 & W - I \end{bmatrix}
\boxplus
\begin{bmatrix} - \tfrac{1}{2} L_2^* L_2 & -L_2^*  \\ L_2 & 0 \end{bmatrix}
\in
\mathfrak{i}( \wh{\noise_1} \ot \init ) \boxplus \mathfrak{u}( \wh{\noise_2} \ot \init )
\]
with
$\re K_1 = - \tfrac{1}{2} L_1^* L_1$,
$W$ isometric
and
$\noise^{W}_{\Pres} = \noise_1$;
$W$ being unitary if and only if $F \in \Qulg$.
\end{rems}

%%%%%%%%%%%%%%%%%%%%%%%%%%%%%%%%%%%%%%%%%%%%%%%%%%%
% NEW SECTION
%%%%%%%%%%%%%%%%%%%%%%%%%%%%%%%%%%%%%%%%%%%%%%%%%%%

 \section{Perturbation of QS cocycles}
 \label{section: perturbation}

 In case the second of two elementary QS cocycles
 $V^1$, $V^2$ is isometric,
 there is another, more standard, way of realising the QS cocycle
 $X^{F_1 \comp F_2}$
 where $F_1$ and $F_2$ are the stochastic generators of $V^1$ and $V^2$.

 \begin{thm}[\emph{Cf.}~\cite{EvansHudson}, \cite{BLS}]
 \label{thm: perturbation}
 Let $F_1 \in \Qqclg$ and $F_2 \in \Qilg$
 Then
 \[
 X^{F_1 \comp F_2} =
 X^{j^2, F_1} X^{F_2}
 \]
 where $j^2$ is the normal *-monomorphic QS mapping cocycle
% $(\ad X^{F_2}_t )_{t\ges 0}$
on \tu{(}the von Neumann algebra\tu{)} $B(\init)$ given by
\[
j^2_t(x) = X^{F_2}_t ( x \ot I_\Fock ) ( X^{F_2}_t )^*
\qquad
( x \in B(\init), t \in \Rplus ),
\]
and
 $X^{j^2, F_1}$ is the unique strong solution of the QS differential
 equation
 $\ud X_t = X_t \cdot G_t \ud\Lambda_t$, $X_0 = I$, for the
 integrand process
 $G := \big( (\id_{B(\khat)} \otol j^2_t)(F_1) \big)_{t\ges 0}$.
\end{thm}

\begin{proof}
 Given the existence of $X^{j^2, F_1}$ and its quasicontractivity (\cite{BLS}),
 the result follows easily from the quantum It\^o product formula
 and uniqueness for
weak solutions,
which are bounded with locally uniform bounds,
 of
 the QS differential equation
  $\ud X_t = X_t \cdot (F_1 \comp F_2)  \ud\Lambda_t$, $X_0 = I_{\init \ot \Fock}$
  (see Theorem~\ref{thm: cocycle = qsde}).
\end{proof}

\begin{rem}
In~\cite{BLS} we worked in the equivalent category of
QS \emph{right} cocycles.
\end{rem}

Applying this result to the Gaussian/non-Gaussian decomposition of QS generators
(Theorem~\ref{thm: G decomp})
yields the following result.
Recall the injections~\eqref{eqn: iota con} associated with
 realising the concatenation product in terms of the series product~\eqref{eqn: effectively}.

\begin{cor}
\label{cor: perturbabion}
Let $F \in \Qqclg$ with block matrix form $\left[ \begin{smallmatrix} * & * \\ * & C - I \end{smallmatrix}\right]$.
Then
 \[
X^F =
 X^{j^2, F_1} X^{F_2},
 \]
where,
for the orthogonal decomposition
$\noise \ot \init = ( \noise^C_{\Pres} \otimes \init )  \op  ( \noise^{C}_{\Gauss} \otimes \init )$,
$$
F_1 = \iota (F_{\wnGauss})
\ \text{ and } \
F_2 = \iota' (F_{\mGauss}).
$$
\end{cor}

\begin{proof}
The inclusions
\[
\iota' \big( \QpglgSolo ( \wh{\noise^C_{\Gauss}} \ot \init) \big)
\subset
\Qpglg
\subset
\Qulg
\]
ensure that $X^{F_2}$ is unitary and so Theorem~\ref{thm: perturbation} applies.
\end{proof}

%%%%%%%%%%%%%%%%%%%%%%%%%%%%%%%%%%%%%%%%%%%%%%%%%%%
% NEW SECTION
%%%%%%%%%%%%%%%%%%%%%%%%%%%%%%%%%%%%%%%%%%%%%%%%%%%

\section{Holomorphic QS Cocycles}
 \label{section: holomorphic cocycles}

 In this section the setting is extended to holomorphic QS cocycles (\cite{LS holomorphic}).
 Before formulating the conjecture, the corresponding It\^o algebra is investigated, mirroring
 Section~\ref{section: quantum Ito algebra}.
 As is customary, we identify each bounded Hilbert space operator
 $T$ with its associated quadratic form $q_T$,
 given by $q_T[\xi] := \ip{\xi}{T\xi}$.
We also use the notation $q(\cdot, \cdot)$
 for the sesquilinear form
 associated with a quadratic form $q[\, \cdot \,]$
 by polarisation.

 Fix Hilbert spaces $\hil$ and $\Hil$.
 Let $\QF$ denote the class of
 quadratic forms $\Gamma$ on $\hil \op \Hil$
 having the following structure:
\[
 %\begin{align*}
% &
\left\{
\begin{array}{l l}
\Dom \Gamma = \Domain \op \Hil
&
 \\
 \Gamma [ \xi] =
 \gamma[u] -
 \big[
 \ip{\zeta}{Lu} +
 \ip{\wt{L}u}{\zeta} +
 \ip{\zeta}{(C-I)\zeta}
 \big],
&
 \
 \text{ for } \xi = \binom{u}{\zeta} \in \Dom \Gamma,
 \end{array} \right.
\]
 where $\Domain$ is a subspace of $\hil$,
 $C \in B(\Hil)$, $\gamma$ is a quadratic form on $\hil$,
 $L$ and $\wt{L}$ are operators from $\hil$ to $\Hil$,
 and
 \[
 \Dom \gamma = \Dom L = \Dom \wt{L} = \Domain.
 \]
For reasons which will become apparent,
$\Domain$ is \emph{not} assumed to be dense in $\hil$.

 Write $\Gamma \sim (\gamma, L, \wt{L}, C)$, and
 refer to $(\gamma, L, \wt{L}, C)$ as the \emph{components} of $\Gamma$.
 Also define an associated operator on $\hil \op \Hil$ by
 \[
  \FD_\Gamma := \begin{bmatrix} 0 & 0 \\ L & C-I \end{bmatrix}.
 \]
 Thus
 \begin{equation*}
% \label{eqn: range}
  \Dom \FD_\Gamma = \Dom \Gamma
  \text{ and }
 \Ran \FD_\Gamma \subset \{ 0 \} \op \Hil
=
\Ran \Delta
 \end{equation*}
where, as usual, $\Delta := 0_\hil \op I_\Hil$.
 The inclusion obviously implies that
 \begin{equation}
 \label{eqn: range}
 \Ran  \FD_\Gamma \subset \Domain' \op \Hil
 \text{ for \emph{any} subspace }
 \Domain'
 \text{ of }
 \hil.
  \end{equation}
 Note that if $\Gamma \in \QF$ with components $\components$,
 then the adjoint form $\Gamma^*$ belongs to $\QF$ too, with
 \begin{equation*}
%  \label{eqn: Gamma dagger}
 \Gamma^* \sim (\gamma^*, \wt{L}, L, C^*)
 \text{ and }
 \FD_{\Gamma^*} =
 \begin{bmatrix} 0 & 0 \\ \wt{L} & C^*-I \end{bmatrix}.
  \end{equation*}
 Thus, in terms of the associated sesquilinear form,
 \begin{equation}
  \label{Gamma Delta xi xi}
 \ip{\xi}{\FD_\Gamma \xi} = - \Gamma (\Delta \xi, \xi)
 \text{ and }
 \ip{\FD_{\Gamma^*} \xi}{\xi} =
 - \Gamma (\xi, \Delta \xi)
 \qquad
 (\xi \in \Dom \Gamma).
 \end{equation}

 \begin{defn}
 For $\Gamma_i \in \QF$, with components $\componentsi$
 ($i=1,2$),
 define
 $\Gamma_1 \comp \Gamma_2,\, \Gamma_1 \Delta \Gamma_2 \in \QF$
 by
 \begin{align*}
 &\Gamma_1 \comp \Gamma_2 \sim
 \components, \text{where}
 \\
 &\qquad \quad
 \gamma[u] =
 \gamma_1[u] + \gamma_2[u] - \ip{\wt{L}_1 u}{L_2 u},
 \\
 &\qquad \quad
 L = L_1 + C_1 L_2, \
 \wt{L} = C_2^* \wt{L}_1 + \wt{L}_2 \
 \text{ and } \
 C = C_1 C_2;
 \\
 &\Gamma_1 \Delta \Gamma_2 \sim
 \components, \text{where}
 \\
 &\qquad \quad
 \gamma[u] =
 - \ip{\wt{L}_1 u}{L_2 u},
 \\
 &\qquad \quad
 L = (C_1 - I) L_2, \
 \wt{L} = (C_2^* - I) \wt{L}_1 \
 \text{ and } \
 C = (C_1 - I) (C_2 - I) +I.
 \end{align*}
 \end{defn}
 Thus
 \begin{subequations}
  \begin{align}
  \label{6.3a}
 &\Gamma_1 \comp \Gamma_2 =
 \Gamma_1 + \Gamma_2 + \Gamma_1 \Delta \Gamma_2,
 \\
 \nonumber
 &\Dom ( \Gamma_1 \comp \Gamma_2 ) =
 \Dom ( \Gamma_1 \Delta \Gamma_2 ) =
 \Dom \Gamma_1 \cap \Dom \Gamma_2,
 \\
  \label{F = CF}
 &( \Gamma_1 \Delta \Gamma_2 )[\xi] =
 - \ip{\FD_{\Gamma^*_1} \xi}{\FD_{\Gamma_2} \xi}
 \quad
 ( \xi \in \Dom \Gamma_1 \cap \Dom \Gamma_2),
 \text{ and }
 \\
  \label{5.3now}
 &\FD_{\Gamma_1 \Delta \Gamma_2} =
\big(  \{ 0 \} \op (C_1 - I) \big)
 %\begin{bmatrix} 0 & \\ & C_1-I \end{bmatrix}
 \FD_{\Gamma_2}.
 \end{align}
 \end{subequations}

 \begin{lemma}
 \label{lemma: Gamma Delta algebra}
  The prescription
  $(\Gamma_1, \Gamma_2) \mapsto \Gamma_1 \Delta \Gamma_2$
  defines an associative and bilinear composition on the vector space $\QF$
 which is also involutive:
 \begin{equation}
  \label{Gamma Delta star}
 ( \Gamma_1 \Delta \Gamma_2 )^* =
 \Gamma_2^* \Delta \Gamma_1^*.
 \end{equation}
 \end{lemma}
 \begin{proof}
 Let $\Gamma_i \in \QF$ with domain $\Domain_i \op \Hil$
 ($i=1,2,3$).
 Bilinearity follows from the evident linearity of the map
 $\Gamma \mapsto \FD_\Gamma$, and~\eqref{Gamma Delta star}
 holds since, for $\xi \in (\Domain_1 \cap \Domain_2) \op \Hil$,
 \[
 ( \Gamma_1 \Delta \Gamma_2 )^* [\xi] =
 \ol{( \Gamma_1 \Delta \Gamma_2 ) [\xi]} =
 - \ip{\FD_{\Gamma_2} \xi}{\FD_{\Gamma^*_1} \xi} =
 ( \Gamma_2^* \Delta \Gamma_1^* ) [\xi].
 \]
Clearly
$\Dom \big( ( \Gamma_1 \Delta \Gamma_2 ) \Delta \Gamma_3 \big) =
\big( \Domain_1 \cap \Domain_2 \cap \Domain_3 \big) \op \Hil =
\Dom \big(  \Gamma_1 \Delta ( \Gamma_2  \Delta \Gamma_3  )\big)$
and,
for
 $\xi \in (\Domain_1 \cap \Domain_2 \cap \Domain_3) \op \Hil$,~\eqref{5.3now}
implies that
\begin{align*}
 \ip{\FD_{\Gamma_1^*} \xi}{\FD_{\Gamma_2\Delta \Gamma_3} \xi}
 &=
  \ip{\FD_{\Gamma_1^*} \xi}
  {\big( \{0\}\op(C_2-I) \big) \FD_{\Gamma_3} \xi}
 \\
 &=
 \ip{\big( \{0\}\op(C_2^*-I) \big)\FD_{\Gamma_1^*} \xi}
  {\FD_{\Gamma_3} \xi}
  \\
  &=
  \ip{\FD_{\Gamma_2^* \Delta \Gamma_1^*} \xi}
  {\FD_{\Gamma_3} \xi}
  =
 \ip{\FD_{(\Gamma_1 \Delta \Gamma_2)^*} \xi}
  {\FD_{\Gamma_3} \xi}.
\end{align*}
 Thus $\Delta$ is associative, by~\eqref{F = CF}.
 \end{proof}

 \begin{propn}
 \label{propn: Gamma series semigroup}
 The composition
  $(\Gamma_1, \Gamma_2) \mapsto \Gamma_1 \comp \Gamma_2$
  endows $\QF$ with the structure of a *-monoid
  whose identity element is $\Gamma_0 \sim ( 0, 0, 0, I )$, in particular,
  \begin{equation*}
 % \label{Gamma series star}
 ( \Gamma_1 \comp \Gamma_2 )^* =
 \Gamma_2^* \comp \Gamma_1^*.
  \end{equation*}
 \end{propn}

 \begin{proof}
 Let $\Gamma_1, \Gamma_2, \Gamma_3 \in \QF$.
 In view of Lemma~\ref{lemma: Gamma Delta algebra},
 \[
 ( \Gamma_1 + \Gamma_2 + \Gamma_3 ) +
 ( \Gamma_1 \Delta \Gamma_2 +
   \Gamma_2 \Delta \Gamma_3 +
   \Gamma_1 \Delta \Gamma_3 ) +
   \Gamma_1 \Delta  \Gamma_2 \Delta \Gamma_3
 \]
 is a common expression for
 $\Gamma_1 \comp ( \Gamma_2 \comp \Gamma_3 )$ and
 $( \Gamma_1 \comp  \Gamma_2 ) \comp \Gamma_3 $,
which have common domain
$ (\Domain_1 \cap \Domain_2 \cap \Domain_3) \op \Hil$.
It is easily seen that the element
$\Gamma_0 \sim ( 0, 0, 0, I )$
satisfies
$\Gamma_0 \comp \Gamma = \Gamma = \Gamma \comp \Gamma_0$ for all $\Gamma \in \QF$.
 The fact that the adjoint operation defines an involution on the resulting
 monoid follows from its additivity on $\QF$
 and the identity~\eqref{Gamma Delta star}.
 \end{proof}
 For the following lemma,
 recall the range observation~\eqref{eqn: range}.

 \begin{lemma}
 \label{lemma: 3 series}
 Let $\Gamma_i \in \QF$ with domain $\Domain_i \op \Hil$
 \tu{(}$i=1,2,3$\tu{)}.
 Then
 \[
 ( \Gamma_1 \comp \Gamma_2 \comp \Gamma_3 )[\xi] =
 ( \Gamma_1 \comp \Gamma_3 )[\xi] +
 \Gamma_2 \big(
 (I + \FD_{\Gamma_1^*}) \xi, (I + \FD_{\Gamma_3}) \xi
 \big)
 \qquad
 (\xi \in \Domain_1 \cap \Domain_2 \cap \Domain_3).
 \]
 \end{lemma}
 \begin{proof}
 Let $\xi \in \Domain_1 \cap \Domain_2 \cap \Domain_3$.
 Then, since
 \[
 \Gamma_1 \comp \Gamma_2 \comp \Gamma_3 - \Gamma_1 \comp \Gamma_3 =
 \Gamma_2 +
 \Gamma_2 \Delta \Gamma_3 +
 \Gamma_1 \Delta \Gamma_2 +
 \Gamma_1 \Delta \Gamma_2 \Delta \Gamma_3,
 \]
 the lemma follows by
 several applications of the identities~\eqref{Gamma Delta xi xi} and~\eqref{5.3now}:
 \begin{align*}
 \Gamma_2 &\big(
 (I + \FD_{\Gamma_1^*} ) \xi,
 (I + \FD_{\Gamma_3} ) \xi
 \big)
 \\
 &=
 \Gamma_2[\xi] +
 \Gamma_2 \big( \xi, \FD_{\Gamma_3}  \xi \big) +
 \Gamma_2 \big( \FD_{\Gamma_1^*}  \xi, \xi \big) +
 \Gamma_2 \big( \FD_{\Gamma_1^*}  \xi, \FD_{\Gamma_3} \xi \big)
 \\
 &=
 \Gamma_2[\xi] -
 \ip{\FD_{\Gamma_2^*} \xi}{\FD_{\Gamma_3} \xi} -
 \ip{\FD_{\Gamma_1^*} \xi}{\FD_{\Gamma_2} \xi} -
 \ip{\FD_{\Gamma_1^*} \xi}{\big(\{0\}\op (I-C_2)\big)\FD_{\Gamma_3} \xi}
 \\
 &=
 \Gamma_2 [\xi] +
 \big( \Gamma_2 \Delta \Gamma_3 \big) [\xi]+
 \big( \Gamma_1 \Delta \Gamma_2 \big) [\xi] +
 \big( \Gamma_1 \Delta \Gamma_2 \Delta \Gamma_3 \big) [\xi].
  \end{align*}
 \end{proof}

 \begin{propn}
 \label{propn: quasi pres}
 Let
 $\Gamma_1, \Gamma_2 \in \QF$
 and $\beta_1, \beta_2 \in \Real$,
 and set
 $\Gamma = \Gamma_1 \comp \Gamma_2$  and $\beta = \beta_1 + \beta_2$.
 \begin{alist}
 \item
Suppose that,
for $i=1,2$,
 $$
\Gamma_i^* \comp \Gamma_i \ges 2 \beta_i \Delta^\perp
 \text{ on }  \Dom \Gamma_i.
$$Then
 $\Gamma^* \comp \Gamma \ges 2 \beta \Delta^\perp$
 on $\Dom \Gamma$.
 \item
Suppose that,
for $i=1,2$,
 $$
\Gamma_i^* \comp \Gamma_i = 0
 \text{ on }  \Dom \Gamma_i.
$$
 Then
 $\Gamma^* \comp \Gamma = 0$
 on $\Dom \Gamma$.
 \end{alist}
 \end{propn}
 \begin{proof}
 By associativity and Lemma~\ref{lemma: 3 series},
 \[
 ( \Gamma^* \comp \Gamma ) [\xi] =
 ( \Gamma_2^* \comp \Gamma_2 ) [\xi] +
 ( \Gamma_1^* \comp \Gamma_1 ) \big[ (I + \FD_{\Gamma_2}) \xi \big]
 \]
 for all $\xi \in \Dom \Gamma_1 \cap \Dom \Gamma_2$.
The result therefore follows since
$\Delta^\perp \FD_{\Gamma_2} = 0$ on $\Dom \Gamma_2$.
 \end{proof}

\begin{rem}
Thus
\[
\big\{
\Gamma \in \QF : \,
\Gamma^* \comp \Gamma = 0 = \Gamma \comp \Gamma^*
\big\}
\]
forms a subgroup of the group of invertible elements of $( \QF, \comp )$.
\end{rem}

To complete the discussion of the algebra of the
 series product on quadratic forms, here is the form generalisation of Proposition~\ref{propn: wills id}.

 \begin{propn}
 \label{propn: QF id}
 Let $\Gamma \in \QF$.
 Then
 \begin{equation}
 \label{eqn: QF id}
 ( \Gamma^* \comp \Gamma ) [ \xi ]
 =
 ( \Gamma \comp \Gamma^* ) \big[ ( I + F^\Delta_\Gamma ) \xi \big]
 +
 \big\| F^\Delta_{\Gamma^* \comp \Gamma} \xi \big\|^2
 \qquad
 (\xi \in \Dom \Gamma).
 \end{equation}
 Let $\mathcal{V} \in \QF$ be of the form
 $\nu \oplus 0_\Hil$ where $\nu \in \mathcal{Q}( \hil )_\sa$.
 Then
 \[
 \Gamma^* \comp \Gamma \ges \mathcal{V}
 \ \text{ if and only if } \
  \Gamma \comp   \Gamma^* \ges \mathcal{V}.
 \]
 \end{propn}

 \begin{proof}
Let $\xi =  \binom{u}{\zeta} \in \Dom \Gamma$.
Note that $\Gamma$, $\Gamma^*$,
$\Gamma^* \comp \Gamma$,
$\Gamma \comp \Gamma^*$
and
$\Gamma^* \comp \Gamma \comp \Gamma^* \comp \Gamma$
all share the same domain.
On the one hand,
setting $\Gamma_1 = \Gamma_2 = \Gamma^* \comp \Gamma$ in~\eqref{6.3a} and~\eqref{F = CF}
yields
\[
\big( \Gamma^* \comp \Gamma \comp \Gamma^* \comp \Gamma \big) [ \xi ]
-
\big( \Gamma^* \comp \Gamma \big) [ \xi ]
=
\big( \Gamma^* \comp \Gamma \big) [ \xi ]
-
\big\| F^\Delta_{\Gamma^* \comp \Gamma} \xi \big\|^2.
\]
On the other hand,
setting
$\Gamma_1 = \Gamma^*$, $\Gamma_2 = \Gamma \comp \Gamma^*$ and $\Gamma_3 = \Gamma$ in
Lemma~\ref{lemma: 3 series}
yields
\[
\big( \Gamma^* \comp \Gamma \comp \Gamma^* \comp \Gamma \big) [ \xi ]
-
\big( \Gamma^* \comp \Gamma \big) [ \xi ]
=
\big( \Gamma \comp \Gamma^* \big) \big[ ( I + F^\Delta_\Gamma ) \xi \big].
\]
Thus~\eqref{eqn: QF id} holds.

Now suppose that $\Gamma \comp \Gamma^* \ges \mathcal{V}$.
Then, since
\[
\mathcal{V} \big[ ( I + F^\Delta_\Gamma ) \xi \big] = \nu[u] = \mathcal{V}[\xi],
\]
\eqref{eqn: QF id} implies that
$( \Gamma^* \comp \Gamma ) [ \xi ] \ges \mathcal{V}[\xi]$.
Thus $\Gamma^* \comp \Gamma \ges \mathcal{V}$.
The converse implication follows by exchanging $\Gamma$ and $\Gamma^*$.
 \end{proof}

 Now we return to the Hilbert spaces $\init$ and $\noise$.
 Let $\Xhol(\init)$ denote the class of
 quadratic forms $\gamma$ on $\init$ which are
 \emph{closed}, \emph{densely defined} and satisfy the
\emph{accretive} and \emph{semisectorial} conditions
 \begin{align*}
 &\re \gamma + \beta \ges 0
\ \text{ and } \
\big|\! \im \gamma[u] \big| \les
 \alpha \big( \re \gamma[u] + \norm{u}^2 \big)
 \qquad
 (u \in \Dom \gamma)
 \end{align*}
 for some $\beta \in \Real$ and $\alpha \in \Rplus$,
 and let $\Shol(\init)$ denote the class of holomorphic
 semigroups they generate (see \emph{e.g.}~\cite{Ouhabaz}).

 In~\cite{LS holomorphic},  a quasicontractive QS cocycle
 $V$ on $\init$ is called \emph{holomorphic}
 if its expectation semigroup belongs to $\Shol(\init)$.
 Denoting this class of QS cocycle by $\QSChol(\init,\noise)$,
 it is shown there that
 the correspondence
 $\Xhol(\init) \to \Shol(\init)$
 extends to a bijection
 \[
 \Xfourhol(\init,\noise) \to \QSChol(\init,\noise),
 \quad
 \Gamma \mapsto X^\Gamma,
 \]
 in which
 $\Xfourhol(\init,\noise)$ denotes the subclass of
 $\mathcal{Q}(\khat\ot\init) = \mathcal{Q}( \init \op ( \noise\ot\init) )$
 consisting of forms $\Gamma \sim \components$ such that
 $\gamma \in \Xhol(\init)$ and $\Gamma^* \comp \Gamma + 2 \beta \Delta^\perp \ges 0$,
 for some $\beta \in \Real$.
 We speak of the \emph{stochastic form generator} of the holomorphic cocycle.
If $\Gamma \sim \components \in \Xfourhol(\init,\noise)$ then
it follows from Proposition~\ref{propn: QF id} that  $\Gamma^* \in \Xfourhol(\init,\noise)$
and in~\cite{LS holomorphic}
it is also shown that
 $C \in B(\noise\ot\init)$ is a contraction,
 and
 $X^{\Gamma^*} = ( X^{\Gamma} )^{\dual}$, the dual QS cocycle
 defined in~\eqref{eqn: dual}.

 The bijection extends the above form-semigroup correspondence as
 follows:
 if $\Gamma \sim (\gamma, 0, 0, I)$
 where
  $\gamma \in \Xhol(\init)$
  then
  $\Gamma \in \Xfourhol(\init,\noise)$
  and
  $X^\Gamma = (P_t \ot I_\Fock)_{t\ges 0}$
 where $P$ is the holomorphic semigroup with form generator $\gamma$.
 It also extends that of Theorem~\ref{thm: cocycle = qsde},
 in the sense that
if
$F \in \Qqclg$
with block matrix form
$
\left[
\begin{smallmatrix} K & M \\ L & C-I
\end{smallmatrix}
\right]
 $
then
 $X^F = X^\Gamma$
for the form in $\Xfourhol(\init,\noise)$ given by
$$
\Gamma \sim (q_{-K}, L, M^*, C).
$$

 If $\Gamma_1, \Gamma_2 \in \Xfourhol(\init,\noise)$
 then $\Gamma_1 \comp \Gamma_2 \in \Xfourhol(\init,\noise)$
 provided only that
 $\Dom \Gamma_1 \cap \Dom \Gamma_2$ is dense.
 This neatly extends the fact that
 if $\gamma_1, \gamma_2 \in \Xhol(\init)$
 then $\gamma_1 + \gamma_2 \in \Xhol(\init)$
  provided only that
 $\Dom \gamma_1 \cap \Dom \gamma_2$ is dense in $\init$.

 \begin{conjecture}
 Let
 $V^1$ and $V^2$ be quasicontractive holomorphic QS cocycles on $\init$
with noise dimension space $\noise$
and respective stochastic form generators $\Gamma_1$ and $\Gamma_2$,
and suppose that
 $\Dom \Gamma_1 \cap \Dom \Gamma_2$ is dense in $\khat \ot \init$.
 Then
\tu{(}in the notation of Theorem~\ref{thm: the theorem}\tu{)},
for all $T \in \Rplus$, $\varphi \in B(\Fock)_*$ and $u \in \init$,
\[
 \sup_{0 \les r \les t \les T}
 \big\|
 \big( \id_{B(\init)} \otol \varphi \big)
 \big( V^{ \mathcal{P}, 1 \comp 2 }_{r,t} - X^{\Gamma_1 \comp \Gamma_2}_{r,t} \big)
 \, u
 \big\|
 \to 0
 \ \text{ as } \
  \big|\mathcal{P} \cap [0,T] \big| \to 0.
\]
Moreover, if the QS cocycle  $X^{\Gamma_1 \comp \Gamma_2}$ is isometric
and $V^1$ and $V^2$ are contractive
then,
for all $T \in \Rplus$ and $\xi \in \init \ot \Fock$,
\[
 \sup_{0 \les r \les t \les T}
 \big\|
  \big( V^{ \mathcal{P}, 1 \comp 2 }_{r,t} - X^{\Gamma_1 \comp \Gamma_2}_{r,t} \big)
 \, \xi
 \big\|
 \to 0
 \ \text{ as } \
 \big|\mathcal{P} \cap [0,T] \big| \to 0.
\]
 \end{conjecture}
 \begin{rems}
 The conjecture has three special cases where it is proven.
 Theorem~\ref{thm: the theorem} covers the case where
 $\Gamma_1$ and  $\Gamma_2$ are bounded.
 In the semigroup case,
 where $\Gamma_i \sim (\gamma_i, 0, 0, I)$
 for $\gamma_i \in \Xhol(\init)$ ($i=1,2$),
 it reduces to a version of a celebrated result of Kato --
 as extended by Simon (\cite{Kato T}).
 For the case of independent driving noises
 a version of
 the holomorphic counterpart to
 Corollary~\ref{cor: LS L-T},
  which includes the Kato--Simon theorem,
 is proved in~\cite{LS T-K}.
 \end{rems}
%%%%%%%%%%%%%%%%%%%%%%%%%%%%%%%%%%%%%%%%%%%% SECTION

%%%%%%%%%%%%%%%%%%%%%%%%%%%%%%%%%%%      ACKNOWLEDGEMENTS
  \par\bigskip\noindent
  {\bf Acknowledgements.}
I am grateful to Mateusz Jurczy\'nski and Micha\l\ Gnacik for useful comments on an earlier draft of the paper.
 Support from the UK-India Education and Research Initiative
 (UKIERI),
 under the research collaboration grant
 \emph{Quantum Probability, Noncommutative Geometry \& Quantum
 Information},
 is also gratefully acknowledged.

%%%%%%%%%%%%%%%%%%%%%%%%%%%%%%%%%%%      BIBLIOGRAPHY

%   $ deliberate mistake


\begin{thebibliography}{BLS}

\bibitem [Bel] {Belavkin}
 V.P. Belavkin,
 A new form and a *-algebraic structure of
 quantum stochastic integrals in Fock space,
 \emph{Rend. Sem. Mat. Fis. Milano}
 \textbf{58} (1988), 177--193.

   \bibitem [BGL] {BGL}
 A.C.R. Belton, M. Gnacik and J.M. Lindsay,
Strong convergence of quantum random walks via semigroup decomposition,
\textit{arXiv}:1712.02848 [math.PR].

  \bibitem [BLS] {BLS}
 A.C.R. Belton, J.M. Lindsay and A.G. Skalski,
 Quantum Feynman-Kac perturbations,
 \textit{J. London Math. Soc.} (\textit{2})
 \textbf{89} (2014) no. 1, 275--300.

 \bibitem[Bha] {Bhat}
 B.V.R. Bhat,
 Cocycles of CCR flows,
 \emph{Mem. Amer. Math. Soc.}
 \textbf{149} (2001), no. 709.

\bibitem[BrR] {BrR}
O.~Bratteli and D.W.~Robinson,
``Operator Algebras and Quantum Statistical Mechanics II:
Equilibrium states. Models in quantum statistical mechanics,''
2nd Edition, Springer-Verlag, Berlin, 1997.

 \bibitem[Che] {Chernoff}
 P. R. Chernoff,
 Product formulas, non-linear semigroups and
 addition of unbounded operators,
 \emph{Mem. Amer. Math. Soc.}
 \textbf{140}, 1974.

 \bibitem [DGS] {DGS}
 B. Das, D. Goswami and K.B. Sinha,
A homomorphism theorem and a Trotter product formula for
quantum stochastic flows with unbounded coefficients,
\emph{Comm. Math. Phys.}
\textbf{330} (2014) no. 2, 435--467.

 \bibitem [DLT] {DLT}
 B.K. Das, J.M. Lindsay and O. Tripak,
 Sesquilinear quantum stochastic analysis in Banach space,
 \emph{J. Math. Anal. Applic.}
\textbf{409} (2014)  no. 2, 1032--1051.

  \bibitem[Dav] {Davies}
 E.B.~Davies,
 ``One-Parameter Semigroups,''
 Academic Press, London, 1980.

 \bibitem [EvH] {EvansHudson}
 M.P. Evans and R.L. Hudson,
 Perturbations of quantum diffusions,
 \textit{J. London Math. Soc.} (\textit{2})
 (1990), no. 2, 373--384.

 \bibitem [Fag] {Fagnola}
 F. Fagnola,
 Characterization of isometric and unitary weakly differentiable
 cocycles in Fock space,
 \emph{in},
 ``Quantum Probability \& Related Topics,'' QP-PQ VIII,
 (Ed. L. Accardi),
 World Scientific, Singapore, 1993,
 pp. 143--164.

 \bibitem [$\text{GL}\!+$] {GLSW}
 D. Goswami, J.M. Lindsay, K.B. Sinha and S.J. Wills,
 Dilation of Markovian cocycles on a von Neumann algebra,
 \emph{Pacific J. Math.} \textbf{211} (2003) no. 2, 221--247.

 \bibitem [GoJ] {GoughJames}
 J. Gough and M. James,
 The series product and
 its application to quantum feedforward and feedback networks,
 \emph{IEEE Trans. Automat. Control}
 \textbf{54} (2009) no. 11, 2530--2544.

 \bibitem [HiP] {HiP}
 E. Hille and R.S. Philips,
 ``Functional Analysis and Semigroups,''
 \emph{Coll. Publ.} \textbf{31},
 American Mathematical Society, Providence, R.I., 1957.

\bibitem [Hol] {Holevo}
 A.S. Holevo,
 Stochastic representation of quantum dynamical semigroups (Russian),
 [Translated in
 \emph{Proc. Steklov Math. Inst.}
  (1992) no. 2, 145--154.]
  \emph{Trudy Mat. Inst. Steklov}
 \textbf{191} (1989), 130--139.

 \bibitem [HuP] {HuP}
 R.L. Hudson and K.R. Parthasarathy,
 Quantum It\^o's formula and stochastic evolution,
 \emph{Comm. Math. Phys.}
 \textbf{93} (1984) no. 3, 301--323.

  \bibitem [Jou] {Journe}
  J.-L. Journ\'e,
  Structure des cocycles markoviens sur l'espace de Fock,
  \emph{Probab. Theory Related Fields}
  \textbf{75} (1987) no. 2, 291--316.

  \bibitem
 %[$\text{Ka}_2$]
 [Kat]
 {Kato T}
  % --- ---,
 T. Kato,
 Trotter's product formula for an arbitrary pair of self-adjoint
 contraction semigroups,
 \emph{in}, ``Topics in Functional Analysis
 (essays dedicated to M.G. Krein on the occasion of his 70th
 birthday),''
 \emph{Adv. in Math. Suppl. Stud.} \textbf{3},
 Academic Press, London 1978, pp. 185--195.

 \bibitem [L]  {L}
 J.M. Lindsay,
 Quantum stochastic analysis --- an introduction,
 \emph{in},
 ``Quantum Independent Increment Process, I:
 From Classical Probability to Quantum Stochastic Calculus''
 (Eds. U. Franz \& M. Sch\"urmann),
 \emph{Lecture Notes in Mathematics} \textbf{1865},
 Springer-Verlag, Heidelberg 2005,
  pp. 181--271.

% \bibitem [$\text{L}_2$]  {L S-M}
%  --- ---,
% Quantum stochastic integrals and semimartingales,
% \emph{Preprint}.

 \bibitem
 [$\text{LS}_1$] {LS L-T}
 J.M. Lindsay and K. B. Sinha,
 A quantum stochastic Lie--Trotter product formula,
  \emph{Indian J. Pure Appl. Math.}
 \textbf{41} (2010) no. 1, 313--325.

 \bibitem[$\text{LS}_2$] {LS holomorphic}
  --- --- ,
 Holomorphic quantum stochastic contraction cocycles,
 \emph{Preprint}.

 \bibitem[$\text{LS}_3$] {LS T-K}
 --- --- ,
 Trotter--Kato product formulae for quantum stochastic cocycles,
 \emph{in preparation}.

 \bibitem [$\text{LW}_1$]
{LWqsde}
 J.M. Lindsay and S.J. Wills,
 Existence, positivity, and contractivity for
 quantum stochastic flows with infinite dimensional noise,
 \emph{Probab. Theory Rel. Fields}
 \textbf{116} (2000) no. 4, 505--543.

 \bibitem [$\text{LW}\!_2$]
{LW2}
  --- --- ,
Markovian cocycles on operator algebras, adapted to a Fock filtration,
\emph{J. Funct.\ Anal.}\
\textbf{178} (2000), no.\ 2, 269--305.

 \bibitem [$\text{LW}_3$]%[LiW]
 {LiW}
   --- ---,
 Quantum stochastic cocycles and
 completely bounded semigroups on
 operators spaces,
 \emph{Int. Math. Res. Notices}
  (2014) no. 11, 3096--3139.

 \bibitem [Mey] {Meyer}
 P.-A. Meyer,
 ``Quantum Probability for Probabilists''
 (2nd Edn.),
 \emph{Lecture Notes in Math.} \textbf{1538},
 Springer-Verlag, Berlin 1995.

 \bibitem [Ouh] {Ouhabaz}
 E.M. Ouhabaz,
 ``Analysis of Heat Equations on Domains,''
 \emph{London Mathematical Society Monographs} \textbf{31},
 Princeton University Press,
 Princeton, 2005.

 \bibitem [Par] {Partha}
 K.R. Parthasarathy,
 ``An Introduction to Quantum Stochastic Calculus,''
 \emph{Monographs in Mathematics} \textbf{85},
 Birkh\"auser,
 Basel 1992.

  \bibitem [PaS] {PaS}
 K.R. Parthasarathy and K.B. Sinha,
  A random Trotter--Kato product formula,
 ``Statistics and Probability:
 Essays in Honor of C.R.~Rao,''
 \emph{eds.\ G.~Kallianpur, Paruchuri R.~Krishnaiah
 \& J.K.~Ghosh}, North-Holland, Amsterdam,
 1982, pp.\ 553-566.

 \bibitem [ReS] {ReedSimon}
 M. Reed and B. Simon,
 ``Methods of Modern Mathematical Physics, I: Functional Analysis
 (2nd Edn.), II: Fourier Analysis, Self-Adjointness,''
 Academic Press, New York, 1980, 1975.


\bibitem [SSS] {SSS}
L.\ Sahu,
M.\ Sch\"urmann and
K.B.\ Sinha,
Unitary processes with independent increments and representations of Hilbert tensor algebras,
 \emph{Publ.\ Res.\ Inst.\ Math.\ Sci.}
 \textbf{45} (2009) no. 3, 745--785.

\bibitem [Sch] {Schurmann}
M.\ Sch\"urmann,
``White noise on bialgebras,''
\emph{Lecture Notes in Math.}\ \textbf{1544},
 Springer-Verlag, Berlin 1993.

 \bibitem [Sko] {Skorohod}
A.V.\ Skorohod,
``Asymptotic Methods in the Theory of Stochastic Differential Equations,''
\emph{Translations of Mathematical Monographs}
\textbf{78},
American Mathematical Society, Providence, 1989.


 \bibitem[Tro] {Tr2}
 H.F.\ Trotter,
 On the product of semi-groups of operators,
 \emph{Proc.\ Amer.\ Math.\ Soc.}
 \textbf{10} (1959), 545--551.

 \bibitem[Wil] {Wills}
 S.J. Wills,
 On the generators of quantum stochastic operator cocycles,
 \emph{Markov Process. Related Fields}
 \textbf{13} (2007) no. 1, 191--211.

\end{thebibliography}
 \end{document}